\documentclass[11pt]{article}
\usepackage{graphicx}

\usepackage[utf8]{inputenc}
\usepackage{amsmath,amsfonts,amssymb,amscd,amsthm,txfonts,hyperref}

\newtheorem{theorem}{Theorem}[section]
\newtheorem{proposition}[theorem]{Proposition}
\newtheorem{lemma}[theorem]{Lemma}
\newtheorem{corollary}[theorem]{Corollary}
\newtheorem{remark}{Remark}[section]

\theoremstyle{definition}
\newtheorem{definition}[theorem]{Definition}

\newcommand{\R}{\varmathbb R}

\newcommand{\N}{\varmathbb N}

\newcommand{\E}{\mathcal E}

\newcommand{\arct}{\mbox{Arctan }}

\newcommand{\T}{\varmathbb T}

\newcommand{\pt}{\mbox{PT}}
\newcommand{\lap}{\bigtriangleup}

\title{Consequences of the choice of a particular basis of $L^2(S^3)$ for the cubic wave equation on the sphere and the Euclidian space}
\author{Anne-Sophie de Suzzoni\footnote{University of Cergy-Pontoise, UMR CNRS 8088, F-95000 Cergy-Pontoise, e-mail : \texttt{anne-sophie.de-suzzoni@u-cergy.fr}}}

\begin{document}

\maketitle

\begin{abstract}In this paper, the almost sure global well-posedness of the cubic non linear wave equation on the sphere is studied when the initial datum is a random variable with values in low regularity spaces. The domain is first the 3D sphere, thanks to the existence of a uniformly bounded in $L^p$ basis of $L^2(S^3)$ and then the result is extended to $\R^3$ thanks to the Penrose transform.\end{abstract}

\tableofcontents

\section{Introduction}
 
The first aim of this paper is to extend the result by N. Burq and N. Tzvetkov \cite{twon} on the torus to the sphere of dimension 3. In \cite{twon}, N. Burq and N. Tzvetkov have proved the global well-posedness of the cubic non linear equation when the initial datum is a randomization of some function in the product of Sobolev spaces $H^\sigma(\T^3)\times H^{\sigma - 1}(\T^3)$, $\sigma \geq 0$ and $\T^3$ the torus of dimension 3. 

The probabilistic estimates they use in order to prove their result are due to the fact that the $L^p$ norms of the canonical basis of $L^2(\T^3)$ : $(e^{in.x})_{n\in \N^3}$, are uniformly bounded, whether in $n$ or in $p$.

Here, a result of N. Burq of G. Lebeau is required to go on in the case of the sphere, as a basis of $L^2(S^3)$ has a priori no reason to be uniformly bounded in $L^p$. In \cite{BLinj}, they proved that there exists a basis of $L^2$ uniformly bounded in $L^p$ by $C_p$ and formed by eigenfunctions of the Laplace-Beltrami operator on the sphere. Here, the result of \cite{twon} is extended to the sphere, despite the dependence of the bound of the norms of the basis on $p$.

Let us describe the above-mentioned randomization. Let $(e_{n,k})_{n,k}$ be such a basis of $L^2(S^3)$, that is, such that for all $n,k$

$$||e_{n,k}||_{L^p} \leq C_p\; ,$$
and 

$$-\lap_{S^3}e_{n,k}= n^2 e_{n,k}\; ,$$
let $a_{n,k}$ and $b_{n,k}$ be two sequences of real-valued i.i.d. on a probability space $(\Omega, \mathcal A , P)$ and with large Gaussian deviation estimates, which results from the assumption that there exists $c$ such that for all $\gamma$, the following mean values satisfy :

$$E(e^{\gamma a_{n,k}}),E(e^{\gamma b_{n,k}}) \leq e^{c\gamma^2 }\; ,$$
and finally, let $\lambda_{n,k}$ and $\mu_{n,k}$ be two sequences of complex numbers such that for some $\sigma \geq 0$ :

$$\sum_{n,k}(1+n^2)^\sigma |\lambda_{n,k}|^2 < \infty \mbox{ and } \sum_{n,k}(1+n^2)^{\sigma-1} |\mu_{n,k}|^2 < \infty \; .$$

Then the equation 

$$(\partial_T^2 +1-\lap_{S^3})u + u^3 = 0$$
with initial datum the randomization of $\sum \lambda_{n,k} e_{n,k},\sum \mu_{n,k}e_{n,k}$ defined as

$$u_{|T=0}=u_0 = \sum_{n,k}\lambda_{n,k} a_{n,k}e_{n,k} \; , \; \partial_T u_{|T=0} = u_1 = \sum_{n,k} \mu_{n,k} b_{n,k} e_{n,k} \; ,$$
is globally well-posed.

The measure $\mu$ induced by the couple $(u_0,u_1) \in L^2(\Omega, H^\sigma \times H^{\sigma -1})$ where $\sigma $ is given in the assumptions on $(\lambda_n)_n,(\mu_n)_n$ is very similar to the one introduced by the randomization in \cite{BTranI}.

To phrase it more precisely,

\begin{theorem}\label{fst} There exists a set $E$ of full $\mu$-measure such that for all $(v_0,v_1) \in E$ the Cauchy problem

$$\left \lbrace{\begin{tabular}{ll}
$(\partial_T^2 +1-\lap_{S^3})u +  u^3 = 0$ & \\
$u_{|T=0}=v_0 $ & $\partial_T u_{|T= 0} = v_1$ \end{tabular}} \right. $$
has a unique solution in $U(T)(v_0,v_1) + \mathcal C(\R ,H^1(S^3))$ where $U(T)$ is the flow of the linear equation $\partial_T^2 +1-\lap_{S^3} = 0$.\end{theorem}

Note that the wave operator $\partial_T^2 -\lap_{S^3} $ has been replaced here by $\partial_T^2 +1-\lap_{S^3} $ for convenience with respect to the second part of this paper, dedicated to the cubic non linear wave equation on $\R^3$. However, the proof for the cubic non linear wave equation on the sphere would be similar to the one with this operator.

What is more, if the $\lambda_{n,k}$ (resp. or the $\mu_{n,k}$) are supposed to be such that

$$\sum_{n,k} (1+n^2)^s |\lambda_{n,k}|^2 = + \infty \mbox{ (resp. or } \sum_{n,k} (1+n^2)^{s-1} |\mu_{n,k}|^2 = +\infty \mbox{)}$$
for all $s>\sigma$, and the $a_{n,k}$ and $b_{n,k}$ are complex Gaussian of law $\mathcal N(0,1)$, then the elements of $E$ are almost surely in $H^\sigma \times H^{\sigma -1}$ and almost surely not in $H^s\times H^{s-1}$.

As it appears, this result recalls one of \cite{remoi} on the sphere but without the hypothesis of radial symmetry.

The main idea behind the proof is that with large Gaussian deviation estimates, the solution of the linear equation $U(T)(u_0,u_1)$ is made to belong almost surely to $L^p$ for all $p\geq 1$ which ensures local and then global well-posedness. Indeed, it is the gain on integrability on the initial data that helps to gain regularity on the non linear part (namely, the solution minus $U(T)(u_0,u_1)$) of the solution.

A second issue raised on this paper is the properties of the Penrose transform of the solution. The Penrose transform sends solutions of $(\partial_T^2 +1-\lap_{S^3})u + u^3 = 0$ on the sphere to solution of the cubic non linear wave equation on the Euclidean space $\R^3$. Indeed, the change of variable involved in this transform injects $\R\times \R^3$ into $[-\pi,\pi]\times S^3$ and satisfies nice properties with respect to the d'Alembertian $\partial_t^2 -\lap_{\R^3} $. 

Hence, with a solution of $(\partial_T^2 +1-\lap_{S^3})u + u^3 = 0$ on the sphere, the existence of a solution of the cubic NLW on $\R^3$ is expected.

Nevertheless, the use of the Penrose transform raises three problems : first, the space where the solution lives shall be described, then, so does the space where this solution is unique, and finally, the spaces to which the initial data belong or do not belong should be specified.

Unfortunately, the third matter remains unanswered but the author believes that if the work on the sphere is done with $\sigma = 0$, that is with the initial data on the sphere in $L^2\times H^{-1}$, then the initial data on $\R^3$ should be almost surely in $L^2\times H^{-1}$ when multiplied by $(\left(\frac{2}{1+r^2}\right)^{1/2},(\left(\frac{2}{1+r^2}\right)^{-1/2})$, and almost surely not to $H^s\times H^{s-1}$, when $s> 0$.

Nevertheless, the following theorem holds.

\begin{theorem}\label{snd} There exists a measure $\nu$ on $\mathcal L^2(\R^3)\times \mathcal H^{-1}(\R^3)$ with

$$||g||_{\mathcal L^2} = ||\left(\frac{2}{1+r^2}\right)^{1/2}g||_{L^2} \; , \; ||g||_{\mathcal H^{-1}(\R^3)} = ||\left(\frac{2}{1+r^2}\right)^{-1/2}(1-H_1)^{-1} g ||_{L^2}$$
and

$$H_1 = \left( \frac{1+r^2}{2}\right)^2 \lap_{\R^3} + 3\frac{1+r^2}{2} r \partial_r  + 6\frac{1+r^2}{2}$$
and a set $F$ of full $\nu$-measure such that for all $(g_0,g_1) \in F$, the Cauchy problem : 

$$\left \lbrace{ \begin{tabular}{ll}
$(\partial_t^2 -\lap_{\R^3})f + f^3 =0 $& \\
$f_{|t=0} = g_0$ & $\partial_t f_{|t=0} =g_1$\end{tabular}} \right.$$
has a unique global solution in $L(t)(g_0,g_1) + \mathcal C (\R,H^1(\R^3))$ where $L(t)$ is the flow of the linear wave equation.

Moreover, the solution $f$ satisfies scattering in the sense that for all $q\in ]\frac{18}{5},6]$,

$$||f(t)-L(t)(g_0,g_1)||_{L^q} = O((1+t^2)^{-1/6})$$
when $t\rightarrow \pm \infty$.\end{theorem}

The measure $\nu$ is the image measure of the measure $\mu$ induced by $(u_0,u_1)$ by the Penrose transform.

The main idea of the proof is that for $q\geq 4$, the $L^q$ norm of the solution on $\R^3$ is controlled by the $L^q$ norm of the solution on the sphere, which ensures regularity properties and then uniqueness of the solution. 

\paragraph{Plan of the paper} The section 2 is dedicated to the proof of the global well-posedness of the equation on the sphere. The first two subsections, which are about local theory and global theory on $H^\sigma\times H^{\sigma-1} $ with $\sigma > 0$, are very similar to \cite{twon}. The third one, where the global well-posedness is dealt with the initial datum being almost surely in $L^2\times H^{-1}$ presents divergences, in particular due to the fact that the bound on the $L^p$ norms of the chosen basis of $L^2$ depends on $p$.

The third section is about the Penrose transform and how it acts on the norms of the solution and the norms of the initial data. The transform is presented, along with its trace on the initial data, and then what its trace turns the Laplace-Beltrami operator on the sphere to in order to study the Sobolev norms of the initial data on $\R^3$.

The fourth one is about the uniqueness and scattering properties of the solution of the equation on $\R^3$. It focuses on the integrability of the solution (how a $L^p$ norm of the solution on the sphere is changed by the Penrose transform), then its regularity before stating the uniqueness and scattering results.

\section*{Acknowledgements}

The author would like to thank Nicolas Burq for suggesting the problem.

\section{Almost sure existence of global solutions on the sphere}

This section deals with the global well-posedness for the cubic wave equation on the sphere with initial data taken as a random variable on $H^\sigma \times H^{\sigma-1}$, $\sigma\geq 0$. The solution has a linear part in $\mathcal C(\R, H^\sigma) $ and a non linear part in $\mathcal C (\R , H^1)$.

\subsection{Definition of the initial data and local theory}
Following the techniques of N. Burq and N. Tzvetkov in \cite{twon}, the random initial datum shall be chosen such that when it is submitted to the linear flow of the wave equation, it has $L^p$ norms in time and space.

Theorem 6 of the third section of \cite{BLinj} provides the existence of a Hilbertian basis of $L^2(S^3)$ composed with spherical harmonics uniformly bounded in $L^p$. Let us therefore name the different objects that shall be needed to define a suitable initial datum.

First, let us recall that the eigenvalues of $1-\lap_{S^3}$ are $n^2, n\geq 1$.

Thanks to the result of N. Burq and G. Lebeau, denote by $(e_{n,k})_{n\geq 1, 1\leq k\leq (n+1)^2}$ a Hilbertian basis of $L^2(S^3)$ uniformly bounded in $L^p$.

\begin{theorem}[Burq, Lebeau] There exists a Hilbertian basis $(e_{n,k})_{n,k}$ of $L^2(S^3)$ such that : 

$$(1-\lap_{S^3}) e_{n,k}  = n^2 e_{n,k}$$
for all $n\geq 1, 1\leq k\leq n^2$, $n^2$ being the dimension of the subspace of $L^2$ spanned by the spherical harmonics of degree $n-1$, and such that there exists a constant $C_p$  such that for all $n,k$

$$||e_{n,k}||_{L^{p}(S^3)} \leq C_p \; .$$
\end{theorem}

To be more precise, Burq and Lebeau proved the following proposition. Set $E_n$ the vector space spanned by the spherical harmonics of degree $n-1$ and $U_n$ the set of orthonormal basis of $E_n$.

\begin{proposition}[Burq, Lebeau]\label{appendum} There exist a family of measure $\nu_n$ on $U_n$ with $\nu_n(U_n) = 1$ and two constants $c_0,c_1>0$ such that for all $p\geq 2$, all $n\in \N^*$, and all $\Lambda \geq 0$, the probability : 
$$
\nu_n ( \lbrace (e_{n,k})_{1\leq k\leq (n+1)^2} \in U_n \; | \; \exists k, |\|e_{n,k}\|_{L^p} - M_{n,p}| >\Lambda \rbrace ) 
$$
where $M_{n,p}$ is a real number bounded by $C\sqrt p$ with $C$ independent from $n$ and $p$, that is the probability that the difference between $M_{n,p}$ and the norm of at least one of the functions of the basis is bigger than $\Lambda$ is bounded by : 
$$
c_0e^{-c_1 n^{4/p}\Lambda^2 }n^2 \; .
$$
\end{proposition}

In the Appendix \ref{sec-uniform}, we give a straightforward proof of this Proposition, without using the general framework that Burq and Lebeau used, but largely inspired by their paper, and we deduce from it that there exists a sequence $p_m \rightarrow \infty$ and a basis $e_{n,k}$ of spherical harmonics such that
\begin{equation}\label{uniform}
\|e_{n,k}\|_{L^{p_m}} \leq C\sqrt{p_m} \; .
\end{equation}

As afore-mentioned, the main difference between this section of this paper and the one by by Burq and Tzvetkov \cite{twon} is that in their paper, the basis of $L^2$ is bounded uniformly in $L^p$, but uniformly in terms of $p$ too. This property allows them to ask for an almost sure $L^p$ bound (so to speak) on the initial datum and take $p \rightarrow \infty$. The difference will appear and be detailed later.

Let $\sigma \in [0,\frac{1}{2}[$ and $(u^{n,k}_0)_{n,k}$ and $(u^{n,k}_1)_{n,k}$ be two sequences of real numbers such that the series

$$\sum_{n\geq 1} n^{2\sigma} \sum_{1\leq k\leq (n+1)^2} (u_0^{n,k})^2 \mbox{ and } \sum_{n\geq 1} n^{2(\sigma-1)} \sum_{1\leq k\leq (n+1)^2} (u_1^{n,k})^2$$
converge but at least one of the series

$$\sum_{n\geq 1} n \sum_{1\leq k\leq (n+1)^2} (u_0^{n,k})^2 \mbox{ or } \sum_{n\geq 1} n^{-1} \sum_{1\leq k\leq (n+1)^2} (u_1^{n,k})^2$$
diverges, that is to say

$$\left( \sum_{n,k} u_0^{n,k}e_{n,k} , \sum_{n,k} u_1^{n,k}e_{n,k}\right)$$
belongs to $H^\sigma\times H^{\sigma -1}$ but is not necessarily in the critical space for the cubic NLW $H^{1/2}\times H^{-1/2}$.

Let $(X, \mathcal A , P)$ be a probability space large enough such that two sequences $(a_{n,k})$ and $(b_{n,k})$ of random variables can be taken satisfying that the $a_{n,k}$ are independent from each other and from the $b_{n,k}$, the $b_{n,k}$ are independent from each other and that there exists $a$ such that for all $n,k$ and all $\gamma \in \R$ :

$$E(e^{\gamma a_{n,k}}), E(e^{\gamma b_{n,k}}) \leq e^{a\gamma^2}$$
where $E$ is the mean value with respect to the probability measure $P$, which ensures that the random variables have Gaussian large deviation estimates.

\begin{proposition}The sequences of $L^2(X,H^\sigma (S^3))$ and $L^2(X,H^{\sigma -1}(S^3))$ respectively

$$u_0^N  = \sum_{n=1}^N \sum_{k=1}^{(n+1)^2} u_0^{n,k}a_{n,k}e_{n,k} \mbox{ and } u_1^N = \sum_{n=1}^N \sum_{k=1}^{(n+1)^2} u_1^{n,k}b_{n,k}e_{n,k} $$
converges. Let $u_0$ and $u_1$ their limits.\end{proposition}

\begin{proof} The proof consists in the fact that the mean values of $a_{n,k}^2$ and $b_{n,k}^2$ are uniformly bounded by $8a$. It ensures that the sequences are Cauchy in their respective spaces and therefore that they converge.\end{proof}

Let $U(T)$ be the flow of the linear equation $(\partial_T^2+1-\lap_{S^3}) u = 0$, that is

$$U(T)\left( \sum_{n,k} v^{n,k}_0 e_{n,k}, \sum_{n,k} v_1^{n,k} e_{n,k} \right) = \sum_{n,k} (\cos(nT) v_0^{n,k} + \frac{\sin (nT)}{n} v_1^{n,k})e_{n,k} \; .$$ 

Set

$$S_M^N = \sum_{n=N}^M \sum_k n^{2\sigma}(u_0^{n,k})^2 + n^{2(\sigma -1)}(u_1^{n,k})^2 \;  ,$$

$$S_M = S_M^0\; , S^N = \lim_{M\rightarrow \infty }S_M^{N+1} \mbox{ and } S = S_N+S^N\; .$$

Set also $\Pi_N, N\geq 0$ the orthogonal projection on the subspace of $L^2$ spanned by $\lbrace e_{n,k} \; |\; n\leq N\rbrace$ with the convention $\Pi_0 = 0$.

The initial data $u_0,u_1$ satisfy some properties regarding the spaces where they belong. With $p_m$ the sequence that goes to $\infty$ for which we have the uniform bound \eqref{uniform} : 

\begin{proposition}\label{proba}There exists $C,c>0$ such that for all $\lambda\geq 0$ : \begin{itemize}
\item for all $N\geq 0$, all $p_m$ and with $\delta_{p_m} = \frac{2}{p_m} > \frac{1}{p_m}$

$$P(||\frac{1}{1+|T|^{\delta_{p_m}}}(1-\Pi_N)U(T)(u_0,u_1)||_{L^{p_m}(\R\times S^3)}\geq \lambda ) \leq \left( \frac{C p_m \sqrt{S^N}}{\lambda}\right)^{p_m}\; ,$$
\item with $\delta_3 = 2/3 > 1/3$,

$$P(||\frac{1}{1+|T|^{\delta_3}}U(T)(u_0,u_1)||_{L^3_T,L^6(S^3)}\geq \lambda ) \leq C e^{-c\lambda^2/S} \; $$
\item for all $M\geq 1$ and with $s=1$ if $\sigma = 0$ and $s=0$ otherwise

$$P(||\frac{1}{1+T^2}\Pi_M U(T)(u_0,u_1)||_{L^1_T,L^\infty(S^3)} \geq \lambda ) \leq Ce^{-c\lambda^2/(M^s S)}\; .$$\end{itemize}
\end{proposition}

The difference between the choices for $\sigma$ in the third inequality is due to the fact that if $\sigma > 0$ then by Sobolev embeddings, there exists some $p$ large enough such that $W^{\sigma,p}$ is embedded in $L^\infty$. In the proof, it will appear that $(1+T^2)^{-1}U(T)(u_0,u_1)$ is almost surely and with the same deviation estimates in $L^1(\R,W^{\sigma,p})$, the $L^1$ norm being taken on the time. Hence, when $\sigma > 0$, the bound does not depend on $M$, as we can take the left hand-side of the inequality with $M\rightarrow \infty$. For $\sigma = 0$, we can not apply Sobolev embedding, thus the bound depends on $M$, we chose $s = 1$ but we could have chosen any $s>0$. 

In the proof, we will write $p$ instead of $p_m$. 

\begin{remark}This proposition differs from the similar one in the torus case, \cite{twon} where the first inequality corresponded to : 

$$P(||\frac{1}{1+|T|^{\delta_p}}(1-\Pi_N)U(T)(u_0,u_1)||_{L^p(\R\times S^3)}\geq \lambda ) \leq \left( \frac{C \sqrt p \sqrt{S^N}}{\lambda}\right)^p \; .$$\end{remark}

\begin{proof}
\begin{lemma} There exists $C$ such that for all $q\geq 1$ and all couples of $l^2$ sequences $v_{n,k},w_{n,k}$ :

$$||\sum_{n,k} a_{n,k}v_{n,k}+b_{n,k}w_{n,k}||_{L^q_X} \leq C \sqrt q \left( \sum |v_{n,k}|^2+ |w_{n,k}|^2 \right)^{1/2} \; .$$\end{lemma}

% \begin{proof}Let $v^N$ be the random variable 
% 
% $$v^N  = \sum_{n=1}^N \sum_{k=1}^{(n+1)^2} v_{n,k}a_{n,k}+ w_{n,k}b_{n,k}\; .$$
% 
% For all $t\geq 0$ and all $N\geq 1$,
% 
% $$P(v^N \geq \lambda ) = P(e^{t v^n }\geq e^{t\lambda}) \leq e^{-t\lambda} E(e^{tv^N})$$
% and using the independence of the random variables $a_{n,k}$ and $b_{n,k}$ :
% 
% $$P(v^N \geq \lambda ) \leq e^{-t\lambda} e^{at^2 S(v^N)}$$
% with 
% 
% $$S(v^N) = \sum_{n=1}^N \sum_{k=1}^{(1+n)^2} v_{n,k}^2 + w_{n,k}^2\; .$$
% 
% Taking $t = \frac{\lambda}{2a S(v^N)}$,
% 
% $$P(v^N \geq \lambda ) \leq e^{-\lambda^2/(4a S(v^N))} \; .$$
% 
% To get the probability $P(|v^N |\geq \lambda )$, use the fact that $S(v^N) = S(-v^N)$, that is
% 
% $$P(v^N \leq -\lambda ) = P(-v^N \geq \lambda ) \leq e^{-\lambda^2/(4a S(v^N))}\; .$$
% 
% $$P(|v^N| \geq \lambda ) \leq 2 e^{-\lambda^2/(4a S(v^N))}\; .$$
% 
% Then, with a change of variable $y=\frac{\lambda}{\sqrt {2aS(v^N)}}$ :
% 
% $$||v^N||_{L^q}^q = \int q \lambda^{q-1} P(|v^N|\geq \lambda) d \lambda \leq 2 q\int \lambda^{q-1} e^{-\lambda^2/(4a S(v^N))}d \lambda = 2 q\left(2 a S(v^N)\right)^{q/2}\int y^{q-1}e^{-y^2/2}dy$$
% 
% $$||v^N||_{L^q}\leq C(q) S(V^N)^{1/2}$$
% with
% 
% $$C(q) = \sqrt{2 a}\left( 2q \int y^{q-1}e^{-y^2/2}dy \right)^{1/q}\leq C \sqrt q\; .$$
% 
% Using the convergence of $v^N$ toward $\sum_{n,k} a_{n,k}v_{n,k}+b_{n,k}w_{n,k}$ in $L^2_X$ : 
% 
% $$||\sum_{n,k} a_{n,k}v_{n,k}+b_{n,k}w_{n,k}||_{L^q_X} \leq C \sqrt q \left( \sum |v_{n,k}|^2+ |w_{n,k}|^2 \right)^{1/2} \; .$$\end{proof}

The proof can be found in \cite{BTranI}, Lemma 3.1.

\begin{lemma}\label{probtool}There exists $C$ such that for all $r,p \geq 1$, $s\geq 0$, $M>N\geq 0$ and $q\geq r,p$,

$$||\frac{1}{1+|T|^{2/r}} (1-\Pi_N) (1-\lap)^{s/2} U(T)(u_0^M,u_1^M)||_{L^q_X,L^r_T,L^p(S^3)} \leq C \sqrt p \sqrt q M^{s'} \sqrt{S_M -S_N}$$
with $s' = s-\sigma$ if $s\geq \sigma$ and $s'=0$ otherwise.\end{lemma}

\begin{proof} Let 

$$\Sigma_N^M (x) = \sum_{n = N+1}^M \sum_{k=1}^{1+n^2}n^{2s}\left( (u_0^{n,k})^2 +n^{-2}(u_1^{n,k})^2\right) |e_{n,k}(x)|^2\; .$$

Using the previous lemma and bounding the sines and cosines by $1$ in 

$$(1-\lap)^{s/2} U(T)(u_0^M,u_1^M)\; ,$$
it appears that

$$||\frac{1}{1+|T|^{2/r}} (1-\Pi_N) (1-\lap)^{s/2} U(T)(u_0^M,u_1^M)||_{L^q_X} \leq \frac{C}{1+|T|^{2/r}}\sqrt q \sqrt{ \Sigma_N^M (x)}\; .$$

Hence, as $q\geq r,p$, and thanks to Minkowski inequality, one can reverse the order of the norms : 

$$||\frac{1}{1+|T|^{2/r}} (1-\Pi_N) (1-\lap)^{s/2} U(T)(u_0^M,u_1^M)||_{L^q_X,L^r_T,L^p(S^3)} $$
$$\leq ||\frac{1}{1+|T|^{2/r}} (1-\Pi_N) (1-\lap)^{s/2} U(T)(u_0^M,u_1^M)||_{L^r_T,L^p(S^3),L^q_X}$$

$$\leq C \sqrt q ||\frac{1}{1+|T|^{2/r}}||_{L^r_T} ||\Sigma_N^M ||_{L^{p/2}}^{1/2}\; .$$

The map $\frac{1}{1+|T|^{2/r}}$ is in $L^r$ and its norm is less than some constant which does not depend on $r$ and

$$||\Sigma_N^M ||_{L^{p/2}} \leq \sum_{n = N+1}^M \sum_{k=1}^{1+n^2}n^{2s}\left( (u_0^{n,k})^2 +n^{-2}(u_1^{n,k})^2\right) ||e_{n,k}(x)||_{L^p}^2$$
as $\|e_{n,k}\|_{L^p} \leq C \sqrt p$,
$$
\leq C p \sum_{n = N+1}^M \sum_{k=1}^{1+n^2}n^{2s}\left( (u_0^{n,k})^2 +n^{-2}(u_1^{n,k})^2\right)\leq C p M^{2s'} (S_M-S_N)\; .
$$

Therefore,

$$||\frac{1}{1+|T|^{2/r}} (1-\Pi_N) (1-\lap)^{s/2} U(T)(u_0^M,u_1^M)||_{L^q_X,L^r_T,L^p(S^3)} \leq C \sqrt p \sqrt q M^{s'} \sqrt{S_M -S_N}\; .$$
End of the proof of Lemma \ref{probtool}\end{proof}

To prove the first inequality of the proposition, take $M \rightarrow \infty$, $s=0$, and $r=q=p$ in the previous lemma to get : 

$$||\frac{1}{1+|T|^{2/p}} (1-\Pi_N) U(T)(u_0,u_1)||_{L^p_X,L^p_T,L^p(S^3)} \leq C p \sqrt{S^N}\; .$$

Then, 

$$P(||\frac{1}{1+|T|^{2/p}} (1-\Pi_N) U(T)(u_0,u_1)||_{L^p_T,L^p(S^3)}\geq \lambda) $$
$$= P(||\frac{1}{1+|T|^{2/p}} (1-\Pi_N) U(T)(u_0,u_1)||_{L^p_T,L^p(S^3)}^p\geq \lambda^p)$$

$$\leq \lambda^{-p} E(||\frac{1}{1+|T|^{2/p}} (1-\Pi_N) U(T)(u_0,u_1)||_{L^p_T,L^p(S^3)}^p) $$

$$= \lambda^{-p}||\frac{1}{1+|T|^{2/p}} (1-\Pi_N) U(T)(u_0,u_1)||_{L^p_X,L^p_T,L^p(S^3)}^p$$

$$\leq \left( \frac{Cp\sqrt{S^N}}{\lambda}\right)^p\; .$$

To prove the second, use the previous lemma with $r=3$, $p=6$, $q\geq 6$, $s=0$, $M\rightarrow \infty$, $N=0$ to get : 

$$P(||\frac{1}{1+|T|^{2/3}}  U(T)(u_0,u_1)||_{L^3_T,L^6(S^3)}\geq \lambda)\leq \left( \frac{C\sqrt q \sqrt S}{\lambda}\right)^q\; .$$

For $\lambda \geq \sqrt 6 \sqrt S C e$, choose

$$q = \frac{\lambda^2}{C^2 e^2 S}\geq 6$$
to get

$$P(||\frac{1}{1+|T|^{2/3}}  U(T)(u_0,u_1)||_{L^3_T,L^6(S^3)}\geq \lambda)\leq e^{-p}  = e^{-c \lambda^2/S}$$
and for small $\lambda$ use the fact that the probability is bounded by $1$ which is less than $e^6e^{-c \lambda^2/S}$ to get for all $\lambda$

$$P(||\frac{1}{1+|T|^{2/3}}  U(T)(u_0,u_1)||_{L^3_T,L^6(S^3)}\geq \lambda) \leq C e^{-c \lambda^2/S}\; .$$

For the third inequality with $\sigma = 0$, use the previous lemma with $N=0$, $r=1$, $s=\frac{1}{2}$, $p$ some $p_m> 6$, $q\geq p_m$ to get

$$||\frac{1}{1+T^2} \Pi_M U(T)(u_0,u_1) ||_{L^q,L^1_T,L^\infty (S^3)} $$

$$\leq ||\frac{1}{1+T^2} \Pi_M U(T)(u_0,u_1) ||_{L^q,L^1_T,W^{1/2,p} (S^3)} \leq C M^{1/2}\sqrt q \sqrt{S_M} \leq C M^{1/2}\sqrt q \sqrt{S} $$
thanks in particular to the Sobolev embedding $W^{1/2,7}\rightarrow L^\infty$ and then 

$$P(||\frac{1}{1+|T|^{2}}  \Pi_M U(T)(u_0,u_1)||_{L^1_T,L^\infty(S^3)}\geq \lambda)\leq \left( \frac{C\sqrt q M^{1/2}\sqrt S}{\lambda}\right)^q$$
and finally

$$P(||\frac{1}{1+|T|^{2}}  \Pi_M U(T)(u_0,u_1)||_{L^1_T,L^\infty(S^3)}\geq \lambda) \leq C e^{-c \lambda^2/(MS)}\; .$$

For the third inequality with $\sigma > 0$, use the previous lemma with $N=0$, $r=1$, $s=\sigma$, $p $ some $p_m$ larger than $\frac{4}{\sigma}$, $q\geq p$ to get

$$||\frac{1}{1+T^2} \Pi_M U(T)(u_0,u_1) ||_{L^q,L^1_T,L^\infty (S^3)} $$

$$\leq ||\frac{1}{1+T^2} \Pi_M U(T)(u_0,u_1) ||_{L^q,L^1_T,W^{1/2,p} (S^3)} \leq C \sqrt q \sqrt{S_M} \leq C \sqrt q \sqrt{S}\; , $$
then 

$$P(||\frac{1}{1+|T|^{2}}  \Pi_M U(T)(u_0,u_1)||_{L^1_T,L^\infty(S^3)}\geq \lambda)\leq \left( \frac{C\sqrt q \sqrt S}{\lambda}\right)^q\; ,$$
and finally, with an appropriate choice for $q$,

$$P(||\frac{1}{1+|T|^{2}}  \Pi_M U(T)(u_0,u_1)||_{L^1_T,L^\infty(S^3)}\geq \lambda) \leq C e^{-c \lambda^2/S}\; .$$
\end{proof}

Thanks to previous proposition, it is known now that $\frac{1}{1+|T|^{2/3}}U(T)(u_0,u_1)$ belongs almost surely to $L^3_T,L^6(S^3)$. Let us use this property in the local theory.

First, rewrite the equation on the sphere in a more convenient way.

The map $u$ solves 

\begin{equation}\label{eqsph} \left \lbrace{\begin{tabular}{ll}
$\partial_T^2 u +(1-\lap)u + u^3 = 0 $ & \\
$u_{|T=0} = v_0 $ & $\partial_T u_{|T=0} = v_1$\end{tabular}} \right. \end{equation}
if and only if $v = u - U(T)(v_0,v_1)$ solves, with $g(T) = U(T)(v_0,v_1)$ : 

\begin{equation}\label{eqsphde} \left \lbrace{ \begin{tabular}{ll}
$\partial_T^2 v +(1-\lap)v + (g+v)^3 = 0 $ & \\
$v_{|T=0} = 0 $ & $\partial_T v_{|T=0} = 0$\end{tabular}} \right. . \end{equation}

\begin{proposition} There exists $C$ such that for all $\Lambda > 0$, all $T_0\in \R$ and all $g,v_0,v_1$ such that

$$||\frac{1}{1+|T|^{2/3}}g||_{L^3_T,L^6(S^3)}^3 \leq \Lambda \; , \; ||v_0||_{H^1}\leq \Lambda\; , \; ||v_1||_{L^2}\leq \Lambda \; ,$$
the equation

\begin{equation}\label{eqloc} \left \lbrace{\begin{tabular}{ll}
$\partial_T^2 v +(1-\lap)v + (g+v)^3 = 0 $ & \\
$v_{|T=T_0} = v_0 $ & $\partial_T v_{|T=T_0} = v_1$\end{tabular}} \right. \end{equation}
has a unique solution in $\mathcal C ([T_0-T_1,T_0+T_1],H^1)$ with $T_1 = \min (1, \frac{1}{C\Lambda^2(1+T_0^2)^3})$. \end{proposition}

\begin{proof}Let

$$\phi_{g,v_0,v_1} (v)(T) = S(T-T_0)(v_0,v_1) - \int_{T_0}^{T}\frac{\sin((T-\tau)\sqrt{1-\lap})}{\sqrt{1-\lap}}\left( (g+v)^3(\tau)\right) d\tau\; .$$

The equation \eqref{eqloc} can be rewritten as the fixed point problem $\phi_{g,v_0,v_1} (v)=v$. The map $\phi_{g,v_0,v_1}$ satisfies : 

$$||\phi_{g,v_0,v_1}(v)(T)||_{H^1} \leq  C\Lambda + \int_{T_0}^T ||(g+v)^3||_{L^2}$$

$$||(g+v)^3||_{L^2} = ||g+v||_{L^6}^3 \leq C (||g||_{L^6}^3+ ||v||_{L^6}^3 ) \leq C (||g||_{L^6}^3 + ||v||_{H^1}^3)$$

$$||\phi_{g,v_0,v_1}(v)(T)||_{H^1} \leq  C\left( \Lambda + (1+|T-T_0|^2+|T_0|^2)||\frac{1}{1+|\tau|^{2/3}}g||_{L^3_\tau,L^6(S^3)}+ \int_{T_0}^T ||v(\tau)||_{H^1}^3 d\tau \right) \; .$$

With $T \in [T_0-T_1,T_0+T_1]$,

$$||\phi_{g,v_0,v_1}(v)||_{L^\infty_T,H^1} \leq C\left( (2+T_0^2+T_1^2)\Lambda + |T_1|\; ||v||_{L^\infty_T,H^1}\right)\; .$$

If $T_1 \leq \min ( 1, \frac{1}{C^3\Lambda^2 (4+T_0^2)^3}$ and $||v||_{L^\infty,H^1} \leq C \Lambda (4+T_0^2)$, then 

$$||\phi_{g,v_0,v_1}(v)||_{L^\infty_T,H^1} \leq C(4+T_0^2) \Lambda$$
so the ball of radius $C(4+T_0^2)\Lambda$ in $\mathcal C ([T_0-T_1,T_0+T_1],H^1(S^3))$ is stable under $\phi_{g,v_0,v_1}$.

What is more, in this ball

$$||\phi_{g,v_0,v_1}(v)-\phi_{g,v_0,v_1}(w)||_{L^\infty_T,H^1} $$

$$\leq C||v-w||_{L^\infty_T,H^1} \left( (2+T_0^{2})||\frac{1}{1+|T|^{2/3}}g||_{L^3_T,L^6 }^2||1_{[T_0-T_1,T_0+T_1]}||_{L^{3}_T} + T_1 (||v||_{L^\infty,H^1}^2 + ||w||_{L^\infty_T,H^1}^2)\right)$$

$$\leq  C||v-w||_{L^\infty_T,H^1} \left( T_1^{1/3} (2+T_0^2) \Lambda^{2/3} + T_1 \Lambda^2 (4+T_0^2)^2\right)\; .$$

Therefore with $C$ large enough and $T_1 \leq \frac{1}{C \Lambda^2 (1+T_0^2)}$, the fixed point theorem applies which concludes the proof.\end{proof}

Thanks to the local Cauchy theory, one can see that the solution of \eqref{eqsphde} can be extended for bigger times as long as the energy : 

$$\mathcal E (T) = \int v(1-\lap)v + \int (\partial_T v)^2 + \frac{1}{2}\int v^4$$
is finite.

To bound this quantity, different arguments are used depending on whether the initial data have been built with $\sigma = 0$ or $\sigma > 0$.

\subsection{Global solutions on the sphere : case 1}

\begin{theorem}\label{caseone}Suppose that $\sigma > 0$. There exists a set $E_\sigma \subseteq H^\sigma \times H^{\sigma -1}$ such that the probability

$$P((u_0,u_1) \in E_\sigma) =1$$
and that for all $v_0,v_1 \in E_\sigma$, the Cauchy problem \eqref{eqsph} with initial datum $v_0,v_1$ is globally well-posed in $U(T)(v_0,v_1)+ \mathcal C (\R, H^1)$.\end{theorem}

\begin{proof} The third inequality of Proposition \ref{proba} ensures that, when $\sigma > 0$, $\frac{1}{1+T^2}U(T)(u_0,u_1)$ belongs almost surely to $L^1_T, L^\infty (S^3)$ and $\frac{1}{1+|T|^{2/3}} U(T) (u_0,u_1)$ belongs almost surely to $L^3_T,L^6 (S^3)$. Therefore, take for $E_\sigma$ the set of initial data which satisfy 

$$||\frac{1}{1+T^2}U(T)(v_0,v_1)||_{L^1_T,L^\infty (S^3)} < \infty \; ,$$

$$||\frac{1}{1+|T|^{2/3}} U(T) (v_0,v_1)||_{L^3_T,L^6(S^3)} < \infty \; .$$

For $v_0,v_1 \in E_\sigma$, call $g(T) = U(T)(v_0,v_1)$ and let $v$ be the local solution of 

$$\partial_T^2 v +(1-\lap) v + (g+v)^3 = 0 $$
with initial datum $0,0$.

According to the local Cauchy theory, the solution $v$ exists as long as 

$$\int (\partial_T v)^2 + \int v(1-\lap) v $$
is finite.

Take 

$$\mathcal E^2 (T)= \int (\partial_T v)^2 + \int v(1-\lap ) v + \frac{1}{2} \int v^4 $$
and differentiate this quantity with respect to $T$.

$$\left( \partial_T \mathcal E \right) \mathcal E = \int \partial_T v \partial_T^2 v + \int \partial_T v (1-\lap) v + \int \partial_T v v^3$$

$$ = \int (\partial_T v)\left( v^3 - (g+v)^3) \right) \; .$$

Hence,

$$\partial_T \mathcal E \leq ||v^3 - (g+ v)^3||_{L^2} \leq C \left( ||g(T)^3||_{L^2} + ||g^2v(T)||_{L^2} + ||gv^2||_{L^2} \right)$$

$$|\partial_T \mathcal E| \leq C \left( ||g(T)||_{L^6}^3 + ||g||_{L^6}^2||v||_{L^6} + ||g||_{L^\infty} ||v||_{L^4}^2\right) $$

$$|\partial_T \mathcal E| \leq C \left( ||g(T)||_{L^6}^3 + ||g||_{L^6}^2\mathcal E  + ||g||_{L^\infty} \mathcal E \right) $$
thanks to Sobolev embedding $H^1 \rightarrow L^6$, and applying Gronwall lemma : 

$$\mathcal E (T) \leq C \int_{0}^T ||g(\tau)||_{L^6}^3 d\tau e^{c\int_{0}^T (||g(\tau)||_{L^6}^2 + ||g(\tau)||_{L^\infty}) d\tau} < \infty \; ,$$
the energy is bounded, which concludes the proof of Theorem \ref{caseone}.\end{proof}

\subsection{Global solutions on the sphere : case 2}

\begin{theorem}\label{casetwo}Suppose that $\sigma = 0$. There exists a set $E \subseteq L^2 \times H^{ -1}$ such that the probability

$$P((u_0,u_1) \in E) =1$$
and that for all $v_0,v_1 \in E$, the Cauchy problem \eqref{eqsph} with initial datum $v_0,v_1$ is globally well-posed in $U(T)(v_0,v_1)+ \mathcal C (\R, H^1)$.\end{theorem}

\begin{proposition}\label{guex}Let $T_0 > 0$. There exists $C(T_0)$ such that for all $\theta > 0 $ and $p=\frac{6}{\theta}$, supposing that $g= U(T)(v_0,v_1)$ can be written $g=g_1+g_2$ with 
\small
$$C(T_0)||\frac{1}{1+T^{2/3}}g||_{L^3_T, L^6_x}^3 \leq e^{p/18}\; \mbox{ and } \; C(T_0)(||\frac{1}{1+T^{2/3}}g||_{L^3_T,L^6_x}^2+ ||\frac{1}{1+T^2}g_1||_{L^2_T, L^\infty_x}+||\frac{1}{1+T^{2/p}}g_2||_{L^p_T,x}) \leq \frac{p}{18}$$
\normalsize then the equation \eqref{eqloc} has a unique solution onto $\mathcal C ([-T_0,T_0], H^1)$. The constant $C$ depends on $T_0$ but is independent of $ \theta$.\end{proposition}

\begin{proof}We consider the energy given in the previous subsection : 

$$\mathcal E(v)^2 = \int v(1-\lap )v +\int (\partial_T v)^2 +\frac{1}{2} \int v^4 \; .$$

If $v$ is the local solution of \eqref{eqloc}, on its interval of definition, it comes : 

$$\partial_T \mathcal E (v) \mathcal E(v) = \int (\partial_T v) \left( g^3 +3g^2 v + 3(g_1+g_2)v^2\right)$$
and by using H\"older inequalities, in particular on the last term, with $1/p'+1/p = 1/2$, 
$$
\Big| \int (\partial_T v) g_2v^2 \leq \|\partial_T v\|_{L^2} ||g_2(T)||_{L^p}||v^2||_{L^{p'}}
$$
we get
$$\partial_T \mathcal E (v) \leq ||g(T)||_{L^6_x}^3 + 3 ||g(T)||_{L^6_x}^2 ||v||_{L^6} + 3 ||g_1(T)||_{L^\infty}||v^2||_{L^2}+ 3||g_2(T)||_{L^p}||v^2||_{L^{p'}}\; .
$$

Then, by using Sobolev embedding $H^1 \subset L^6$ and because the $H^1$ norm of $v$ is controlled by the energy

$$||v||_{L^6} \leq C ||v||_{H^1} \leq C \mathcal E$$
and because the $L^4$ norm to the square is controlled by the energy :
$$||v^2||_{L^2} = ||v||_{L^4}^2 \leq C \mathcal E $$
finally, as $\theta = \frac{6}{p}$,
$$
\frac{1}{2p'} = \frac{1}{4}-\frac{1}{2p} = \frac{1}{4} - \frac{\theta}{12} = \frac{1-\theta}{4} + \frac{\theta}{6}
$$
we get
$$||v^2||_{L^{p'}} = ||v||_{L^{2p'}}^2 \leq  (||v||_{L^4}^{1-\theta} ||v||_{L^6}^\theta)^2 \leq C \mathcal E^{1-\theta} ||v||_{H^1}^{2\theta} \leq C \mathcal E(v)^{1+\theta}\; .$$

Thus,

$$\partial_T \mathcal E (v) \leq ||g(T)||_{L^6_x}^3 + C\left( (||g(T)||_{L^6_x}^2  +  ||g_1(T)||_{L^\infty})\mathcal E+ ||g_2(T)||_{L^p}\mathcal E^{1+\theta} \right) \; .$$

As $\mathcal E$ is continuous and initially $0$, suppose that until time $T_1$ it is less than $e^{p/6}=e^{1/\theta}$, then until time $T_1$, it appears that : 

$$\partial_T \mathcal E(v) \leq ||g(T)||_{L^6_x}^3 + C\left( (||g(T)||_{L^6_x}^2  +  ||g_1(T)||_{L^\infty})+ ||g_2(T)||_{L^p}\right) \mathcal E \; .$$

Using Gronwall lemma,
\small
\begin{eqnarray*}
\E(v)  & \leq C(1+T_0^2)||g||_{L^3_T,L^6_x}^3 e^{C((1+T_0^2)^{2/3}||\frac{1}{1+T^{2/3}}g||_{L^3_T,L^6_x}^2+(1+T_0^2)||\frac{1}{1+T^2}g_1||_{L^1_T,L^\infty}+(1+T_0^2)^{(1+p)/2p}||\frac{1}{1+T^{2/p}}g_2||_{L^p_{T,x}})}\\
 & \leq C(1+T_0^2)||g||_{L^3_T,L^6_x}^3 e^{C(1+T_0^2) \left( ||\frac{1}{1+T^{2/3}}g||_{L^3_T,L^6_x}^2 + ||\frac{1}{1+T^2}g_1||_{L^1_T,L^\infty}+||\frac{1}{1+T^{2/p}}g_2||_{L^p_{T,x}}\right)} \; .
\end{eqnarray*}
\normalsize

Choosing $C(T_0) = C(1+T_0^2)$, by hypothesis : 

$$\E (v) \leq e^{p/9}<e^{p/6} \; .$$

Suppose that the solution $v$ is not well posed on $[-T_0,T_0]$, then as $\mathcal E (v)$ controls the $H^1$ norm of $v$ and the $L^2$ norm of $\partial_T v$, there exists a time $T_1$ such that for all time $T$ smaller than $T_1$, the energy $\mathcal E(v)$ is smaller than $e^{p/6}$ and a $\epsilon$ such that for all $T\in ]T_1,T_1+\epsilon[$, $\mathcal E (v) > e^{p/6}$. Then, thanks to the previous computation, until $T_1$, the energy is strictly less than $e^{p/6}$ and as is it continuous, there exists $\epsilon'$ such that the energy remains smaller than $e^{p/6}$ until $T_1+\epsilon'$ with contradicts the hypothesis.

Hence, the equation \eqref{eqsphde} has a unique solution in $\mathcal C([-T_0,T_0], H^1)$ provided that $g$ satisfies the right properties.\end{proof}

\begin{definition}Let $\theta \in \lbrace \frac{6}{p_m},m\in \N \rbrace$, $p=\frac{6}{\theta}$. As 
$$
S^{N} = \sum_{n > N} (u_0^{n,k})^2 + n^{-1} (u_1^{n,k})^2
$$
converges toward $0$ when $N$ goes to $\infty$, there exists $N(T_0)$ such that $\sqrt{S^{N(T_0)}} $ is smaller than $\frac{1}{54 e C(T_0)C_1}$, where $C_1$ is the constant involved in the first inequality of Proposition \ref{proba} and $C(T_0)$ is the one involved in Proposition \ref{guex}, let

$$F_\theta (T_0)= \lbrace v_0,v_1 \; | \; C(T_0)||U(T)(v_0,v_1)||_{L^3_T,L^6_x}^3 \leq e^{p/18} \rbrace \; , $$

$$G_\theta (T_0)= \lbrace v_0,v_1 \; | \; C(T_0)||U(T)(v_0,v_1)||_{L^2_T,L^6_x}^2 \leq \frac{p}{54} \rbrace \; ,$$

$$H_\theta (T_0)= \lbrace v_0,v_1 \; | \; C(T_0)||U(T)\Pi_N (v_0,v_1)||_{L^1_T,L^\infty_x} \leq \frac{p}{54 } \rbrace\; , $$

$$I_\theta (T_0)= \lbrace v_0,v_1 \; | \; C(T_0)||U(T)(1-\Pi_N)(v_0,v_1)||_{L^p_{T,x}} \leq \frac{p}{54 }\rbrace \; ,$$

$$J_\theta (T_0) = F_\theta\cap G_\theta\cap H_\theta \cap I_\theta \; .$$

Call then

$$E (T_0)= \bigcup_{\theta \in \lbrace \frac{6}{p_m},m\in \N \rbrace} J_\theta \; .$$
\end{definition}

\begin{remark}The separation between the high and low frequencies is useful there, as $S^N$ can be taken as small as one wants and ensures that the measure of $I_\theta^c$ is small enough. \end{remark}

\begin{proposition} The set $E(T_0)$ is of full $\mu$-measure.\end{proposition}

\begin{proof}The measures of the complementary of the different sets defined satisfy : 

$$\mu ( F_\theta^c ) = \mu \left( \lbrace v_0,v_1 \; | \; ||U(T)(v_0,v_1)||_{L^3_T,L^6_x} > e^{p/54} \rbrace \right) \leq C e^{-c(T_0) e^{p/27}} $$

$$\mu (G_\theta^c ) \leq \mu \left( \lbrace v_0,v_1 \; | \; ||U(T)(v_0,v_1)||_{L^3_T,L^6_c} > \sqrt{\frac{ p}{54 C}} \rbrace \right) \leq Ce^{-c(T_0) p}$$

$$\mu (H_\theta^c) = \mu \left( \lbrace v_0,v_1 \; | \; ||U(T)\Pi_N(v_0,v_1)||_{L^1_T,L^\infty_x} > \frac{p}{54C} \rbrace \right) \leq C e^{-c(T_0) p^2/ N}$$

$$\mu (I_\theta^c) = \mu \left( \lbrace v_0,v_1 \; | \; ||U(T)(1-\Pi_N)(v_0,v_1)||_{L^p_{T,x}} > \frac{p}{54 C} \rbrace \right) \leq \left( \frac{C_1 p 54 C(T_0) \sqrt{S^N}}{p}\right)^p$$

$$\mu (I_\theta^c)\leq e^{-p} \; .$$

It comes :

$$\mu (J_\theta^c) \leq C e^{-c(T_0) p} \; .$$

Thus, for all $ \theta$, $E(T_0)$ satisfies

$$\mu(E^c(T_0)) \leq \mu (J_\theta^c) \leq Ce^{-c6/\theta} \; .$$

Taking the limit when $\theta$ goes to $0$ (as when $m \rightarrow  \infty$, $p_m \rightarrow \infty$ and then $\frac{6}{p_m} \rightarrow 0$): 

$$\mu(E^c(T_0)) = 0 \; , \; \mu(E(T_0)) =1\; .$$ \end{proof}

\begin{proposition}For all $(v_0,v_1) \in E(T_0)$, the cubic non linear wave equation on the sphere \eqref{eqsph} with initial datum $v_0,v_1$ has a unique solution in $U(T)(v_0,v_1) + \mathcal C([-T_0,T_0], H^1)$. \end{proposition}

\begin{proof}The equation \eqref{eqsphde} with $g = U(T)(v_0,v_1) = g_1+g_2$, $g_1 = \Pi_N g$, $g_2= (1-\Pi_N) g$ is equivalent to \eqref{eqsph} and satisfies the hypothesis of Proposition \ref{guex} for some $\theta \in ]0,1[$, hence it is well posed in $\mathcal C([-T_0,T_0],H^1)$. Thus, \eqref{eqsphde} is well-posed in $U(T)(v_0,v_1)+\mathcal C([-T_0,T_0],H^1)$.\end{proof} 

\begin{definition}Let $T_N$ be an increasing sequence of $\R$ going to $+\infty$ and let 

$$E = \limsup E(T_N) \; .$$ \end{definition}

\begin{proposition} The set $E$ is of full $\mu$-measure. \end{proposition}

\begin{proof}Indeed, using Fatou's lemma,

$$\mu (E^c) = \mu (\liminf E(T_N)^c) \leq \liminf \mu (E(T_N)^c) = 0 \; .$$\end{proof}

\begin{proof} of Theorem \ref{casetwo}. Let $T \geq 0$. As the sequence $T_N$ is increasing toward $\infty$ there exists $N_0$ such that for all $N\geq N_0$, 

$$T_N \geq T_{N_0} \geq T\; .$$

Since $E = \limsup E(T_N)$, for all $N_1$ there exists $N\geq N_1$ such that $v_0,v_1 \in E(T_N)$. With $N_1=N_0$ there exists $N\geq N_0$ such that 

$$T_N\geq T \mbox{ and } v_0,v_1 \in E(T_N) \; .$$

Hence the equation has a unique solution on $U(\tau)(v_0,v_1)+\mathcal C ([-T_N,T_N],H^1)$ and thus in $U(\tau)(v_0,v_1) + \mathcal C([-T,T],H^1)$. Therefore, this property holding for all time $T$, the equation has a unique solution in $U(T)(v_0,v_1)+ \mathcal C(\R,H^1)$.\end{proof}

\section{Reduction to the sphere and almost sure solutions on the Euclidean space}
In this section, the problem on the Euclidean space is reduced thanks to the Penrose transform to the problem on the sphere. The existence of solution for the Cauchy problem with initial data on a suitable space is derived in this way. Note that for all the sequel $\sigma = 0$.

\subsection{Penrose transform and reduction to the sphere}

As a basis of $L^2$ uniformly bounded in $L^p$ is required to use the techniques developed by N. Burq and N. Tzvetkov in \cite{twon} and according to \cite{BLinj}, the problem needs to be reduced to the sphere. For that, the Penrose transform seems appropriate, since it turns the d'Alembertian of $\R^3$ into the d'Alembertian of $S^3$ added to the identity on distributions. 

\begin{definition}[Penrose Transform on the variables]For all $t\in \R$ and $r\in \R^+$, define $T(t,r)$, $R(t,r)$, $R_0(r)$, $\Omega(t,r)$ and $\Omega_0(r)$ as : 
\begin{eqnarray*}
T = \arct (t+r) +  \arct(t-r) \; , \; R = \arct(t+r) - \arct(t-r) \; , \\
\; R_0 (r) = R(0,r) = 2\arct(r) \; , \\
 \Omega(t,r) =\cos T +\cos R =  \frac{2}{\sqrt{(1+(t+r)^2)(1+(t-r)^2)}} \\
 \Omega_0(r) = \Omega(0,r) = \frac{2}{1+r^2}\; .
\end{eqnarray*}
\end{definition}

\begin{proposition} The map

$$t,r,\omega \in \R\times \R^+ \times S^2 \mapsto T(t,r), R(t,r), \omega$$
is a bijection from $\R \times \R^3$ to $S = \lbrace (T,R,\omega) \; |\; \cos T + \cos R > 0 \rbrace$ and its inverse is given by

$$T,R, \omega \mapsto t = \frac{\sin T}{\cos T + \cos R}\; ,\; r = \frac{\sin R}{\cos T + \cos R} \; ,\;\omega\; .$$
\end{proposition}

See \cite{penrose,bonjam} for the proof.

\begin{remark} The map $r,\omega \mapsto 2\arct(r),\omega$ is a bijection from $\R^3$ to $S^3$ deprived of one of its poles, $R_0 (r) = 2\arct(r)\in [0,\pi[$ being the third angle describing a point in $S^3$. \end{remark}

\begin{definition}[Penrose Transform on distributions] Let $f$ be a distribution on $\R \times \R^3$ and $(f_0,f_1)$ be a pair of distributions on $\R^3$. Define then $v = \pt (f)$ the distribution on $S$ and $(v_0,v_1) = \pt_0(f_0,f_1)$ the pair of distributions  on $S^3$ deprived of one of its poles such that 

$$v(T,R,\omega) = f(\frac{\sin T}{\cos T + \cos R},\frac{\sin R}{\cos T + \cos R} , \omega ) (\cos T + \cos R)^{-1} $$
and

$$v_0(R, \omega ) = \frac{f_0(\tan (R/2),\omega)}{1+\cos R} \; , \; v_1(R,\omega) = \frac{f_1(\tan (R/2), \omega)}{(1+\cos R)^2} \; .$$

% In other words, for all $\phi$ $\mathcal C^\infty$ compactly supported in $S$ and all $\phi_0, \phi_1$ $\mathcal C^\infty$ compactly supported in $S^3_*$ :
% 
% $$\langle v , \phi \rangle_{S} = \langle f , \phi(T(t,r),R(t,r),\omega) \Omega^3(t,r)\rangle_{\R \times \R^3}\; ,$$
% 
% $$\langle v_0 , \phi_0\rangle_{S^3_*} = \langle f_0, \phi_0(2\arct r,\omega) \Omega_0(r)^2 \rangle_{\R^3}\; ,$$
% 
% $$\langle v_1 , \phi_1\rangle_{S^3_*} = \langle f_1, \phi_0(2\arct r,\omega) \Omega_0(r)\rangle_{\R^3}\; .$$
\end{definition}

\begin{remark}The definition of $\pt_0$ may appear a little awkward but the idea hidden behind the notations is that $f$ solves the cubic non linear wave equation with initial datum $(g_0,g_1)$ if an extension of $\pt (f)$ solves the equation of the first section with initial datum an extension to $S^3$ of $\pt_0(g_0,g_1)$.\end{remark}

\begin{definition}Let $u$ be a distribution on $\R \times S^3$ and $v_0,v_1$ two distributions on $S^3$, the inverse Penrose transform is given by :

$$\pt^{-1}u (t,r,\omega) = \Omega(t,r) u(\arct(t+r)+\arct(t-r), \arct(t+r)-\arct(t-r),\omega)\; ,$$
which depends only on the restriction of $u$ to $S$ and the inverse Penrose transform at time $t=0 \Leftrightarrow t=0$ by

$$\pt_0^{-1}(r,\omega) (v_0,v_1)=\left( \Omega_0(r)v_0(2\arct(r),\omega)\; ,\; , \Omega_0^2(r) v_1(2\arct(r),\omega)\right)\; ,$$
witch depends only on the restriction on $S^3$ deprived of one of its poles of $v_0,v_1$.\end{definition}

\begin{lemma} If $u$ solves the problem

\begin{equation} \left \lbrace{ \begin{tabular}{ll}
$(\partial_T^2 +1 - \lap_{S^3})u + u^3 = 0 $ \\
$(u_{|T=0},\partial_T u_{|T=0} ) = v_0,v_1$ \end{tabular}}\right. \end{equation}
then the map $f$ defined as the inverse Penrose transform of $u$ restricted to $S$, that is

$$f=\pt^{-1} (u)$$
solves the problem : 

\begin{equation}\label{nlw} \left \lbrace{ \begin{tabular}{ll}
$(\partial_t^2 - \lap_{\R^3} )f + f^3 = 0$ & \\
$f|_{t=0} = g_0 $ & $\partial_t f|_{t=0} = g_1$ \end{tabular}} \right. \end{equation}
where

$$(g_0, g_1) = \pt_0^{-1}(v_0, v_1 )\; .$$
\end{lemma}

\begin{proof} The fact that the Penrose transform sends the action of $\partial_t^2 -\lap_{\R^3}$ on $\R \times \R^3$ onto the action of $\Omega^3(\partial_T^2 +1-\lap_{S^3})$ on $S$ is known and the proof can be found in \cite{penrose}. Thus, on $S$

$$ \left( (\partial_t^2 - \lap_{\R^3} )f + f^3\right) (t,r,\omega)=\Omega^3 (\partial_T^2 +1-\lap_{S^3})\pt(f) + \Omega^3  \pt (f)^3(T(t,r),R(t,r),\omega)= 0\; .$$

What is more, $T=0 \Leftrightarrow t= 0$, 

$$g_0 = f(t=0) = \Omega_0 u(T=0) = \Omega_0 u(R_0(r))$$
and

$$g_1 = \partial_t f(t=0) = (\partial_t \Omega) (t=0) u(0,R_0(r)) + \Omega_0 \partial_t T (t=0) \partial_T u + \Omega_0 \partial_t R (t=0) \partial_R u $$
$$= \Omega_0(r)^2 \partial_T u = \Omega_0^2 v_1(R_0(r))\; .$$ 
\end{proof}

\subsection{Properties of the change of variable}
In this subsection, the properties of the change of variables involved in the Penrose transform is studied, in particular what it implies on operators and norms.

\begin{definition}Let $\Psi$ be the change of variable corresponding to the Penrose transform at time $T=0$, that is to say : 

$$\Psi(v)(r,\omega) = v(2\arct (r),\omega) \; .$$
\end{definition}

\begin{proposition}This change of variable satisfies : \begin{itemize}

\item for all $v,w$, $\int v(R,\omega)w(R,\omega) \sin^2 R d\omega dR  = \int \Psi(v)(r,\omega)\Psi(w)(r,\omega) r^2 \left(\frac{2}{1+r^2}\right)^3 dr$ , 
\item for all $v$, $\int |v|^p \sin^2 R dR = \int |\Psi(v)|^p r^2 \left( \frac{2}{1+r^2}\right)^3 dr$,
\item $\Psi(\lap_{S^3} v) = \left(\frac{1+r^2}{2}\right)^2 \lap_{\R^3}\Psi(v) -\frac{1+r^2}{2}r\partial_r \Psi(v) $.
\end{itemize}
\end{proposition}

\begin{proof}The proof comes from the facts that : \begin{itemize}
\item $dR = \frac{2}{1+r^2}dr$,
\item $\sin R = \frac{2r}{1+r^2}$ and
\item $\tan R = \frac{2r}{1-r^2}$.
\end{itemize}

Hence, to do the change of variable in the integrals, one can use : 

$$\sin^2 R dR  = \left( \frac{2}{1+r^2}\right)^3 r^2 dr\; .$$

The computation of the change of variable on the Laplace-Beltrami operator is quite similar:

\begin{eqnarray*}
\Psi(\partial_R v) & = &(\partial_r R)^{-1} \partial_r \Psi(v)\\
\Psi( \sin^2(R) \partial_R v) &=& \frac{2r^2}{1+r^2}\partial_r\Psi(v)\\
\Psi(\partial_R  \sin^2(R) \partial_R v) & =& (\partial_r R)^{-1} \partial_r \Psi (\sin^2 R \partial_R v) \\
\Psi(\partial_R  \sin^2(R) \partial_R v)  &= &-\frac{2r^3}{1+r^2}\partial_r \Psi(v) + \partial_r \left( r^2 \partial_r \Psi(v)\right)\\
\Psi(\frac{1}{\sin^2 R}\partial_R  \sin^2(R) \partial_R v) &=& -\frac{1+r^2}{2}r\partial_r \Psi(v) + \left(\frac{1+r^2}{2}\right)^2 \frac{1}{r^2}\partial_r (r^2 \partial_r \Psi(v))\\
\Psi(\lap_{S^3} v) &=& \left(\frac{1+r^2}{2}\right)^2 \lap_{\R^3}\Psi(v) -\frac{1+r^2}{2}r\partial_r \Psi(v) .\end{eqnarray*}\end{proof}

\begin{definition}Let $f_0,f_1$ be the inverse Penrose transform at time $T=0$ of $(u_0,u_1)$ and $g_{n,k},h_{n,k}$ the inverse Penrose transform at time $T=0$ of $e_{n,k},e_{n,k}$, that is to say : 

$$f_0 = \sum_{n,k} u_0^{n,k}a_{n,k}g_{n,k} \; , \; f_1 = \sum_{n,k} u_1^{n,k}b_{n,k} h_{n,k} \; .$$
\end{definition}

\begin{proposition}The $g_{n,k}$ are the eigenfunctions of $H_0 = \left( \frac{1+r^2}{2}\right) ^2 \lap_{\R^3} +\frac{1+r^2}{2}r\partial_r  +\frac{3+r^2}{2}$ with eigenvalues $1-n^2$ and the $h_{n,k}$ are the eigenfunctions of $H_1 = \left( \frac{1+r^2}{2}\right)^2 \lap_{\R^3} + 3\frac{1+r^2}{2} r \partial_r  + 6\frac{1+r^2}{2} $ with eigenvalues $1-n^2$, $n\geq 1$.\end{proposition}

\begin{proof}As 

$$g_{n,k} = \frac{2}{1+r^2} \Psi (e_{n,k})$$
they are the eigenfunctions of the operator $H_0$ such that

$$H_0 (g) = \frac{2}{1+r^2}\Psi\left(\lap_{S^3}\Psi^{-1}(\frac{1+r^2}{2} g)\right)\; .$$

It remains to compute $H_0$.

$$H_0 g = \left( \frac{1+r^2}{2}\right) ^2 \lap_{\R^3}g +\frac{1+r^2}{2}r\partial_r g +\frac{3+r^2}{2} g$$

As the $h_{n,k}$ are given by

$$h_{n,k} = \left(\frac{2}{1+r^2}\right)^2 e_{n,k} \; ,$$
they are the eigenfunctions of the operator $H_1$ defined by

$$H_1 h = \left(\frac{2}{1+r^2}\right)^2 \Psi\left( \lap_{S^3} \Psi^{-1} \left(\frac{1+r^2}{2} \right)^2 h\right)\; . $$
By manipulating the expression of $H_1$, we get that
$$H_1 h = \left( \frac{1+r^2}{2}\right)^2 \lap_{\R^3}h + 3\frac{1+r^2}{2} r \partial_r h + 6\frac{1+r^2}{2} h\; .$$

\end{proof}

\begin{corollary} Let $f_0$ and $f_1$ given by $(f_0,f_1) = \pt_0^{-1}(u_0,u_1)$. We have that :

$$||u_0||_{W^{s,p}(S^3)} = ||\left(\frac{2}{1+r^2}\right)^{3/p-1} (1-H_0)^{s/2}f_0||_{L^p(\R^3)}$$

$$||u_1||_{W^{s,p}(S^3)} = ||\left(\frac{2}{1+r^2}\right)^{3/p-2} (1-H_1)^{s/2} f_1||_{L^p(\R^3)}\; .$$
\end{corollary}

\begin{proof}First, do the change of variable in the $L^p$-norm : 

$$||u_0||_{W^{s,p}} = ||\left( \frac{2}{1+r^2}\right)^{3/p} \Psi( (1-\lap_{S^3})^{s/2} u_0)||_{L^p}\; .$$

Then, compute $\Psi( (1-\lap_{S^3})^{s/2} u_0)$ : 

$$\Psi( (1-\lap_{S^3})^{s/2} u_0) = \Psi\left( \sum_{n,k} n^su_0^{n,k}a_{n,k} e_{n,k}\right) = \frac{1+r^2}{2}\sum_{n,k} n^s u_0^{n,k}a_{n,k} g_{n,k}$$

$$ = \frac{1+r^2}{2}\sum_{n,k} (1-H_0)^{s/2} u_0^{n,k}a_{n,k} g_{n,k} = \frac{1+r^2}{2}(1-H_0)^{s/2}f_0\; .$$

In the end, it comes : 

$$||u_0||_{W^{s,p}} = ||\left(\frac{2}{1+r^2}\right)^{3/p-1} (1-H_0)^{s/2}f_0||_{L^p}\; .$$

The second equality is proved the same way. \end{proof}

\subsection{Spaces of definition of the initial data}

Considering the results of the previous subsection, the choice of the random variable $a_{n,k}$ and $b_{n,k}$ will be made such that the initial datum $u_0,u_1$ of the equation reduced to the sphere is a Gaussian vector, in order to state which norms of $u_0$ and $u_1$ and then of the initial datum of the wave equation on the Euclidean space are almost surely finite or infinite.

In this subsection, suppose that $a_{n,k}$ and $b_{n,k}$ not only satisfy the Gaussian large deviation estimate, but that they are Gaussian. To ensure that the initial datum is almost surely not into some spaces, Fernique's theorem should be used : 

\begin{theorem}[Fernique, 1974] Let $X$ be a Gaussian vector with value into a Banach space $B$ and $N$ a pseudo-norm on $B$ (a pseudo-norm has the same properties as a norm except that $\infty$ is one of its possible value), then for all $p\geq 1$ if the mean value of $N(X)^p$ is infinite, then $N(X)$ is almost surely $\infty$ : 

$$E(N(X)^p) = \infty \Rightarrow \; P( N(X) = \infty) = 1 \; .$$
\end{theorem}

For the proof, see \cite{Fint}.

\begin{proposition} The initial datum $u_0,u_1$ is almost surely not in $H^s\times H^{s-1}$ for all $s> 0$.\end{proposition}

\begin{proof} Use Fernique's theorem with $B = L^2$, $p=2$ and $N$ being the $H^s$ norm and $X$ either $u_0$ or $u_1$. As 

$$E(||u_0||_{H^s}^2) = \sum_{n,k} (u_0^{n,k})^2 n^{2s} \; ,$$

$$E(||u_1||_{H^{s-1}}^2) = \sum_{n,k} (u_1^{n,k})^2 n^{2(s-1)} \; ,$$
either one of this series diverges and $u_0$ and $u_1$ are pseudo Gaussian vectors, it comes that almost surely

$$||u_0||_{H^s}=\infty$$
or almost surely

$$||u_1||_{H^{s-1}} = \infty \; .$$
\end{proof}

Considering the remarks on the norms of the initial datum $(f_0,f_1)$ with respect to the ones of $(u_0,u_1)$ in the previous subsection, the author believes that the initial datum $f_0,f_1$ of the cubic non linear wave equation belongs almost surely to $L^2 \times H^{-1}$ with weight $\frac{1}{\sqrt{1+r^2}}$ but is almost surely not in $H^s\times H^{s-1}$ for all $s> 0$. 

Nevertheless, the proof would require the equivalence between the norms 

$$||\left( \frac{2}{1+r^2}\right)^{1/2} (1-H_0)^{s/2}. ||_{L^2(\R^3)} \mbox{ and } ||\left( \frac{2}{1+r^2}\right)^{1/2-s} (1-\lap_{\R^3})^{s/2}. ||_{L^2(\R^3)}$$
in the one hand and

$$||\left( \frac{2}{1+r^2}\right)^{-1/2} (1-H_1)^{s/2}. ||_{L^2(\R^3)} \mbox{ and } ||\left( \frac{2}{1+r^2}\right)^{-1/2-s} (1-\lap_{\R^3})^{s/2}. ||_{L^2(\R^3)}$$
on the other hand.

\begin{definition}Let $\mathcal H_0^s (\R^3)$ and $\mathcal H_1^s(\R^3)$ be the topological spaces defined by the norms 

$$||g||_{\mathcal H_0^s (\R^3)} = ||\left( \frac{2}{1+r^2}\right)^{1/2}(1-H_0)^s g||_{L^2}$$
and

$$||g||_{\mathcal H_1^s (\R^3)} = ||\left( \frac{2}{1+r^2}\right)^{-1/2}(1-H_1)^s g||_{L^2}\; .$$\end{definition}

\begin{proposition}The set $F = \pt_0^{-1} (E)$ is almost surely included in $\mathcal H_0^0 \times \mathcal H_1^{-1}$ and almost surely disjoint from $\mathcal H_0^s \times \mathcal H_1^{s-1}$ for all $s>0$.

Setting $f_0,f_1 = \pt_0^{-1}(u_0,u_1)$, the random variable which is used as initial datum of the cubic NLW on $\R^3$ and $\nu$ the image measure of $\mu$ under $\pt_0$, the set $F$ satisfies $\nu(F^c) = 0$, which means that there exists $\nu$ almost surely a solution of the cubic NLW.  \end{proposition}

\section{Uniqueness of the solution and scattering}

In this section, the uniqueness of the solution, alongside with some scattering properties is proved.

\subsection{Uniqueness}

\begin{theorem} Let $g_0,g_1 \in \pt_0^{-1}(E)$. The Cauchy problem

$$\left \lbrace{\begin{tabular}{ll}
$\partial_t^2 f - \lap f +f^3 = 0$ & \\
$f|_{t=0} = g_0,$ &  $\partial_t f_{|t=0} = g_1$ \end{tabular}} \right. $$
has a unique solution in $L(t) (g_0,g_1)+ \mathcal C (\R,H^1(\R^3))$ where $L(t)$ is the flow of $\partial_t^2 - \lap = 0$. \end{theorem}

\begin{proof}Let $v_0,v_1\in E$ such that $ (g_0,g_1)$ is the inverse Penrose transform of $v_0,v_1$, let $u$ be the solution of the equation on the sphere with initial datum $v_0,v_1$. Let $f$ be the Penrose transform of $u$ restricted to $S$, this map $f$ satisfies the Cauchy problem with initial datum $g_0,g_1$, which gives the existence of the solution. Prove now that this solution is unique.

\begin{lemma}The flow of the linearised around $0$ wave equation is the inverse Penrose transform of the linear flow $U(T)$, that is : 

$$\pt^{-1}U(T)(v_0,v_1) =  (L(t)(g_0,g_1)) \; .$$
\end{lemma}

\begin{proof}The map $w=U(T)(v_0,v_1)$ satisfies

$$\partial_T^2 w + (1-\lap_{S^3}) w = 0$$
with initial datum $v_0,v_1$. Hence its inverse Penrose transform $h$ satisfies

$$\partial_t^2 h -\lap_{\R^3} h = \Omega^3 \Psi \left( \partial_T^2 h +(1-\lap_{S^3}) h \right)  = 0$$
with initial datum $g_0,g_1$, that is, $h = L(t) (g_0,g_1)$.\end{proof}

Then, let $g = f-L(t)(g_0,g_1)$, $g$ is the reverse Penrose transform of $v= u-U(T)(v_0,v_1)$ and is the solution of

$$\partial_t^2 g -\lap g + ( L(t)(g_0,g_1) + g)^3 = 0$$
with initial datum $0,0$, hence it is a fixed point of 

$$\phi(g) = \int_{0}^t \frac{\sin(\sqrt{-\lap}(t-s))}{\sqrt{-\lap}}(L(t)(g_0,g_1)+ g) ds \; .$$

\begin{lemma}Let $w \in L^q([-\pi,\pi]\times S^3) $, $q\geq 4$ and $h$ its reverse Penrose transform, then

$$||h||_{L^q(\R\times \R^3)}\leq C||w||_{L^q([-\pi,\pi]\times S^3)}\; .$$\end{lemma}

\begin{proof} Computing the change of variable $(T,R) = (\arct(t+r)+\arct(t-r),\arct(t+r) - \arct(t-r))$ leads to: 

$$\int_{\R\times \R^3} |h(t,r,\omega)|^q r^2 dr dtd\omega = \int_{\Omega > 0} |\Omega w(R,T,\omega)|^q \Omega^{-4}\sin^2 R dR dT d\omega\; .$$

With $q\geq 4$ and $\Omega = \cos T +\cos R$ being bounded by $2$,

$$||h||_{L^q(\R\times \R^3)}\leq C||w||_{L^q}\; .$$\end{proof}

Therefore, $L(t)(g_0,g_1)$, $g$ and so $f$ belong to $L^6(\R\times \R^3)$. Indeed,

$$||L(t)(g_0,g_1) ||_{L^6(\R\times \R^3)}\leq C||U(T)(v_0,v_1)||_{L^6}$$

$$\leq C\left( \int_{-\pi}^\pi (1+T^{\delta_3})^6\right)^{1/6}||(1+T^{\delta_3})^{-1}U(T)(v_0,v_1)||_{L^3_T,L^6(S^3)}<\infty \; ,$$

$$||g||_{L^6} \leq C ||v||_{L^6} \leq C (2\pi)^{1/6} ||v||_{L^\infty_T,L^6} \leq C ||v||_{L^\infty_T,H^1} < \infty \; .$$

\begin{lemma} The map $g$ belongs to $\mathcal C (\R,H^1(\R^3) )$ and $\partial_t g \in \mathcal C (\R,L^2(\R^3) )$.\end{lemma}

\begin{proof} The map $g$ satisfies 

$$g(\tau) = -\int_{0}^\tau \frac{\sin(\tau -s)\sqrt{-\lap}}{\sqrt{-\lap}} (L(s)(g_0,g_1)+g)^3ds \; .$$

Hence, by differentiating this expression : 

$$\partial_t g = -\int_{0}^\tau \cos((\tau -s)\sqrt{-\lap})(L(s) (g_0,g_1) +g)^3 ds \; ,$$
and using a H\"older inequality on the integral over time :
$$||\partial_t g ||_{L^2(\R^3)} \leq \int_{0}^\tau ||L(s)(g_0,g_1) + g||_{L^6(\R^3)}^3 \leq \sqrt{|\tau|} ||L(s)(g_0,g_1)+ g ||_{L^6(\R\times\R^3)}^3$$
then, we use Minkowski inequality on the $L^2$ norm of $\partial_t g$,
$$||\partial_t g ||_{L^\infty([0,t],L^2(\R^3))} \leq \sqrt{|t|}||L(s)(g_0,g_1)+ g ||_{L^6(\R\times\R^3)}^3\; .$$

Therefore, as $L(s)(g_0,g_1)$ and $g$ belong to $L^6(\R\times \R^3)$, $g$ belongs to $\mathcal C(\R,L^2)$.

Besides,

$$||g||_{\dot H^1} \leq \int_{0}^\tau ||L(s)(g_0,g_1) + g||_{L^6(\R^3)}^3\; .$$

Hence,

$$||g||_{\dot H^1} \leq  \sqrt{|t|}||L(s)(g_0,g_1)+ g ||_{L^6(\R\times\R^3)}^3\; .$$

Therefore, $g$ belongs to $\mathcal C([0,t], H^1(\R^3))$. \end{proof}

Prove now the uniqueness of the solution in $L(t)(g_0,g_1)+\mathcal C ([0, t], H^1(\R^3)$). Let $f_2,f_3$ be two solutions of the cubic wave equation with initial datum $g_0,g_1$, let $h = f_2-f_3$. The map $h$ satisfies : 

$$\partial_t^2 h -\lap h + f_2^3- f_3^3 = 0 \; .$$

Remark that $h$ is in $H^1$ and $\partial_t h$  is in $L^2$. Let 

$$H(t)^2 = \int (\partial_t h)^2 + \int h(1-\lap) h \; .$$

$$2 H'(t) H(t) = -2\int (\partial_t h)\left( -h+ f_2^3- f_3^3\right) = -2 \int (\partial_t h) \left( -h + h(f_2^2 +f_2f_3 +f_3^2 \right)$$

$$H'(t) \leq C||h||_{L^2} + C||h||_{L^6} \left( ||f_2||_{L^6}^2 + ||f_3||_{L^6}^2 \right) $$

As $H(0) = 0$ and

$$\int_{0}^t \left( ||f_2||_{L^6}^2+||f_3||_{L^6}^2 \right) \leq |t|^{2/3}\left(  ||f_2||_{L^6}^2+||f_3||_{L^6_{t,x}}^2 \right)\; ,$$
by Gronwall lemma, $H(t) = 0$ for all time $t$, which proves the uniqueness. \end{proof}

\subsection{Scattering property}

Finally, with those particular initial data for the wave equation, it satisfies a scattering property. More precisely, when $t$ goes to $\pm \infty$ the solution tends to behave like the solution of the linearised around $0$ solution of the equation with same initial datum. This property does not result from a scattering property of the wave equation on the sphere. Indeed, it is the fact that the Penrose transform divides the solution by something that behaves like $\frac{1}{t^2}$ that ensures scattering.

\begin{theorem} Let $q\in]\frac{18}{5},6]$, $(g_0,g_1)\in \pt_0^{-1}(E)$, $f(t)$ the solution of the cubic wave equation with initial datum $g_0,g_1$.

There exists a constant $C$ depending on the initial datum such that

$$||f(t)-L(t)(g_0,g_1)||_{L^q} \leq \frac{C}{(1+t^2)^{1/6}}\; .$$\end{theorem}

\begin{proof}Let $v_0,v_1\in E$ such that $ (g_0,g_1)= \pt_0^{-1}(v_0,v_1)$ and $u$ the solution of \eqref{eqsph} with initial datum $v_0,v_1$.

The map $u$ satisfies : 

$$u(T) - U(T)(v_0,v_1) = -\int_{0}^T \frac{\sin((T-\tau)\sqrt{1-\lap})}{\sqrt{1-\lap}} (u^3(\tau))d\tau \; .$$

Taking the inverse of the Penrose transform of this equality leads to : 

$$f(t)-L(t)(g_0,g_1) = -\Omega (t,r) \left( \int_{0}^{T(t,r)} \frac{\sin(T-\tau)\sqrt{1-\lap})}{\sqrt{1-\lap}}(u^3(\tau)) d\tau\right)(R(t,r))\; .$$

$$||f(t)-L(t)(g_0,g_1)||_{L^q} \leq ||\frac{\Omega}{2\Omega^{2/3} (1+t^2+r^2)^{1/6}}||_{L^{p}} $$
\hspace{2cm}$$||\left( \int_{0}^{T(t,r)} \frac{\sin(T-\tau)\sqrt{1-\lap})}{\sqrt{1-\lap}}(u^3(\tau)) d\tau\right)(R(t,r))2\Omega^{2/3} (1+t^2+r^2)^{1/6}||_{L^6} $$
with $\frac{1}{q} = \frac{1}{p}+ \frac{1}{6}$ ($q\leq 6$).

Let 

$$A = ||\frac{\Omega}{2\Omega^{2/3} (1+t^2+r^2)^{1/6}}||_{L^{p}}$$
and 

$$B = ||\left( \int_{0}^{T(t,r)} \frac{\sin(T-\tau)\sqrt{1-\lap})}{\sqrt{1-\lap}}(u^3(\tau)) d\tau\right)(R(t,r))2\Omega^{2/3} (1+t^2+r^2)^{1/6}||_{L^6} \; .$$

Apply the change of variable $R= \arct(t+r) - \arct(t-r)$ ($t$ is fixed) in $B$. We have :

$$dR = \Omega^2 2(1+t^2+r^2) dr$$

$$\sin^2 R dR = \Omega^4 2(1+t^2+r^2) r^2 dr \; .$$

the quantity $B$ can be rewritten as :

$$B = || \int_{-\pi}^\pi 1_{\tau \leq T(t,R)}\frac{\sin(T-\tau)\sqrt{1-\lap})}{\sqrt{1-\lap}}(u^3(\tau)) d\tau||_{L^6} $$
thus
$$B \leq \int_{-\pi}^\pi ||1_{\tau \leq T(t,R)}\frac{\sin(T-\tau)\sqrt{1-\lap})}{\sqrt{1-\lap}}(u^3(\tau))||_{L^6}d\tau $$

$$B \leq \int_{-\pi}^\pi ||\frac{\sin(T-\tau)\sqrt{1-\lap})}{\sqrt{1-\lap}}(u^3(\tau))||_{L^6}d\tau\; .$$
Then , we use that thanks to Sobolev embedding $H^1 \subset L^6$ :
\begin{eqnarray*}
||\frac{\sin(T-\tau)\sqrt{1-\lap})}{\sqrt{1-\lap}}(u^3(\tau))||_{L^6} & \leq ||\frac{\sin(T-\tau)\sqrt{1-\lap})}{\sqrt{1-\lap}}(u^3(\tau))||_{H^1}\\
 & \leq \|u^3\|_{L^2} = \|u\|_{L^6}^3
\end{eqnarray*}
Hence, we have :
$$B\leq C ||u||_{L^3_{T\in[-\pi,\pi]},L^6}^3 \; .$$

To bound $A$ remark that $\Omega \leq \frac{2}{\sqrt{1+r^2}}$.

$$A\leq \frac{1}{(1+t^2)^{1/6}} ||\Omega^{1/3}||_{L^p} = \frac{1}{(1+t^2)^{1/6}} ||\Omega||_{L^{p/3}}^{1/3}\; .$$

As $q > \frac{18}{5}$,

$$\frac{p}{3} = \frac{2q}{6-q} > 3$$
which ensures that $\Omega \in L^{p/3}$ and bounded uniformly in $t$.

Finally,

$$||f(t)-L(t)(g_0,g_1)||_{L^q} \leq AB \leq \frac{C}{(1+t^2)^{1/6}}\; .$$
\end{proof}

\appendix

\section{Appendix : Uniformly bounded basis}\label{sec-uniform}

In this appendix, we will build a measure $\nu_n$ on the set of orthogonal basis of the space $E_n$ spanned by spherical harmonics of degree $n-1$ (in dimension $3$) that satisfies the property required to take an orthonormal basis of $L^2(S^3)$ that is uniformly bounded in $L^p$.

Let us begin with giving some notations.

For $n \geq 1$, let $(f_{n,k})_{1\leq k\leq N_n}$ be a fixed orthonormal basis of $E_n$, that is, for all $k$
$$
-\lap_{S^3} f_{n,k} = \lambda_n^2 f_{n,k}
$$
with $\lambda_n = \sqrt{n^2 -1}$ and $N_n$ is the dimension of $E_n$, that is $N_n = n^2$.

We identify the functions of $E_n$ whose $L^2$ norm is equal to $1$ with $S_n= S^{N_n-1}$ the unit sphere of $\R^{N_n}$. Call $p_n$ the uniform measure on $S_n$.

We should admit one theorem before we can go on :

\begin{theorem}[see, \cite{Lcon}]\label{lipschitz} If $F : S_n \rightarrow \R$ is Lipschitz continuous on $S_n$ (for the distance $d(u,v) = \|u-v\|_2$) and $M(F)$ is its median defined as :
$$
p_n(F\geq M(F))\geq \frac{1}{2} \mbox{ and } p_n(F\leq M(F)) \geq \frac{1}{2})
$$
then, we have that
$$
p_n(|F-M(F)|> r) \leq 2 e^{-(N_n-1)\frac{r^2}{2\|F\|_{{\rm Lip}}^2}} \; .
$$
\end{theorem}

We then give other definitions.

\begin{definition} We identify the set of all orthonormal basis $(b_{k})_{1\leq k\leq N_n}$ of $E_n$ with the group of all orthogonal operators of $\R^{N_n}$, that is $O(N_n)$, and  we call $\nu_n$ the Haar measure on $O(N_n)$ and $\Pi_{n,k}$ the map that takes a matrix in $O(N_n)$ and gives its $k$-th column.
\end{definition}

\begin{proposition}\label{measures} The image measure of $\nu_n$ by $\Pi_{n,k}$ is equal to $p_n$. \end{proposition}

\begin{proof} For all $R \in O(N_n)$ and for all measurable set $A$, we have by definition : 
$$
\Pi_{n,k}^*\nu_n(R A) = \nu_n (\Pi_{n,k}^{-1} (RA))
$$
but then the fact that $U$ is in $\Pi_{n,k}^{-1} (RA)$ is equivalent to the fact that $R^{-1}\Pi_{n,k}U$ belongs to $A$. We also have that $R$ and $\Pi_{n,k}$ commute :
$$
R^{-1} \Pi_{n,k} U  =  \Pi_{n,k} (R^{-1}U)
$$
hence $\Pi_{n,k}^{-1} (RA) = R \Pi_{n,k}^{-1} (A)$ and thus 
$$
\Pi_{n,k}^*\nu_n(R A) = \nu_n (R \Pi_{n,k}^{-1} (A)) \; .
$$
Then, as $\nu_n$ is the Haar measure of $O(N_n)$ its is invariant through multiplication to the left, so
$$
\Pi_{n,k}^*\nu_n(R A) = \Pi_{n,k}^*\nu_n (A) \; .
$$
Hence $\Pi_{n,k}^*\nu_n$ is invariant through every isometry of $\R^{N_n}$, it is the uniform measure on $S_n$, which is $p_n$.
\end{proof}

To apply Theorem \ref{lipschitz}, the $L^q$ norm has to be Lipschitz continuous on $S_n$. The next Lemma results in this property.

\begin{lemma}There exists $C$ such that for all $u\in E_n$ and $q\in [2,\infty]$, we have that : 
$$
\|u\|_{L^q} \leq C n^{1-2/q} \|u\|_{L^2}\; .
$$
\end{lemma}

\begin{proof} Let $\pi_n$ be the orthogonal projection on $E_n$ and $K$ its kernel, that is, for all $f\in L^2$ and all $x$ :
$$
\int K(x,y)f(y)dy = \pi_n f(x)\; .
$$
This kernel is given by
$$
K(x,y) = \sum_k g_k(x)g_k(y)
$$
for all orthonormal basis $(g_k)_{1\leq k\leq n^2}$ of $E_n$. Hence, as for any rotation $R$ of $S_n$, $(f_{n,k})_k$ and $(f_{n,k}\circ R)_k$ are orthonormal basis of $E_n$, we have :
$$
K(x,y) = \sum_k f_{n,k}(x)f_{n,k}(y) = \sum_k f_{n,k}(Rx)f_{n,k}(Ry) = K(Rx,Rx)\; .
$$
Therefore $K_n(x)$ defined as $\sqrt{K(x,x)}$ is a constant on $S_n$. Let us compute its value. 
$$
K_n(x)^2 = \frac{1}{vol(S^3)} \int K_n(t)^2dt =\frac{1}{vol(S^3)} \int \sum_k |f_{n,k}(t)|^2dt = C N_n= Cn^2\; .
$$
where $C = \frac{1}{vol(S^3)}$ does not depend on $n$.

If $u$ is in $E_n$, with a Cauchy-Schwartz inequality, we get that : 
$$
|u(x)| \leq \|u\|_{L^2} K_n(x) \leq C n \|u\|_{L^2}\; .
$$

Therefore, we have :
$$
\|u\|_{L^\infty} \leq C n \|u\|_{L^2}
$$
and by interpolation,
$$
\|u\|_{L^q} \leq C \|u\|_{L^2}^\theta \|u\|_{L^\infty}^{1-\theta}
$$
with $\theta = \frac{2}{q}$, hence
$$
\|u\|_{L^q} \leq Cn^{1-2/q} \|u\|_{L^2}\; .
$$
\end{proof}

\begin{proposition}\label{prop-lip}With $M_{n,q}$ the median of $\|.\|_{L^q}$ on $S_n$, we have that there exists $c_1 > 0$ such that for all $r$, all $n$ and all $q\geq 2$,
$$
\nu_n \left( \lbrace (b_{k})_{1\leq k \leq N_n} \; | \; \exists k_0\; , \; \Big| \|b_{k_0}\|_{L^q} - M_{n,q} \Big| > r \rbrace \right)\leq 2N_n e^{-c_1 n^{4/q}r^2} \; .
$$
\end{proposition}

\begin{proof}First, we apply the previous Lemma to prove that $\|.\|_{L^q}$ is Lipschitz continuous with Lipschitz constant equal to $Cn^{1-2/q}$ as 
$$
\Big| \|u\|_{L^q} - \|v\|_{L^q}\Big| \leq \|u-v\|_{L^q} \leq Cn^{1-2/q} \|u-v\|_{L^2}\; .
$$

Then, we apply Proposition \ref{measures} to get, with $k_0$ fixed,
$$
\nu_n \left( \Big|\|b_{k_0}\| - M_{n,q}\Big| > r\right) = p_n \left( \Big|\|b\|_{L^q} - M_{n,q}\Big| > r\right)\; .
$$

Finally, we use Theorem \ref{lipschitz} to get
$$
\nu_n \left( \Big|\|b_{k_0}\| - M_{n,q}\Big| > r\right) \leq 2e^{-(N_n-1) \frac{r^2}{2C n^{2-4/q}}} \; .
$$

Since $\frac{N_n-1}{n^{2-4/q}} \geq Cn^{4/q}$ and by summing over $k_0$, we get the result, that is
$$
\nu_n \left( \lbrace (b_{k})_{1\leq k \leq N_n} \; | \; \exists k_0\; , \; \Big| \|b_{k_0}\|_{L^q} - M_{n,q} \Big| > r \rbrace \right)\leq 2N_n e^{-c_1 n^{4/q}r^2} \; .
$$
\end{proof}

Let us now estimate $M_{n,q}$.
% 
% \begin{lemma}Let $\pi$ be the map from $\R^{2N_n}$ to $S_n$ given by :
% $$
% \pi(y_1,\hdots , y_{2N_n}) = \left( \frac{y_1+iy_2}{\|y\|_2},\hdots , \frac{y_{2N_n-1} + i y_{2N_n}}{\|y\|_2}\right)
% $$
% and $\mu$ the measure on $\R^{2N_n}$ given by
% $$
% d\mu(y) = \prod_j \frac{e^{-y_j^2/2}dy_j}{\sqrt{2\pi}}\; .
% $$
% Then, $p_n$ is the image measure of $\mu$ by $\pi$.
% \end{lemma}
% 
% \begin{proof} For all $R \in U(N_n)$, call $g$ the operator in $O(\R^{2N_n})$ given by
% $$
% g_{i,j}= \left \lbrace{ \begin{tabular}{lll}
% $ \Re (R_{k,l})$ & \mbox{ if } $ i=2k,j=2l$ \mbox{ or }$i=2k+1,j=2l+1,$\\
% ${\rm Im}(R_{k,l})$ & \mbox{ if } $i=2k+1,j=2l,$ \\
% $-{\rm Im}(R_{k,l})$ & \mbox{ if } $i=2k,j=2l+1.$ \end{tabular}} \right.
% $$
% We have that $R\pi = \pi g$. We get that for all $A$ measurable,
% $$
% \pi^* \mu (R^{-1}A) = \mu ( \pi^{-1} R^{-1} A) = \mu ( (R\pi)^{-1}A) = \mu (g^{-1}(\pi^{-1}A).
% $$
% But we also have that $g^{-1}$ belongs to $O(\R^{2N_n})$ and $\mu$ is invariant through the action of $O(\R^{2N_n})$ hence $\pi^* \mu$ is invariant through the action of $U(N_n)$ which makes it equal to the uniform measure on $S_n$, that is $p_n$.
% \end{proof}

\begin{lemma}For all $t\in \R_+$, we have
$$
p_n (|x_1|>t) \leq 2 e^{-(N_n-1)\frac{t^2}{2}}\; .
$$
\end{lemma}

\begin{proof}
We use the fact that the projection on the first coefficient of one vector is Lipschitz continuous on the sphere $S_n$ with Lipschitz constant equal to $1$, and that as $p_n$ is uniformly distributed on $S_n$, the median of $x_1$ is $0$, hence, applying Theorem \ref{lipschitz} : 
$$
p_n( |x_1| > t) \leq 2 e^{-(N_n-1)\frac{t^2}{2}}\; .
$$
\end{proof}

% \begin{proof}
% If $t \geq 1$, then 
% $$
% p_n (|x_1| > t ) = 0
% $$
% as $x$ belongs to $S_n$.
% 
% Otherwise, $t=\cos \theta_0$, with $\theta_0 \in ]0,\frac{\pi}{2}]$. We use the description of $p_n$ introduced in the previous lemma to get : 
% $$
% p_n(|x_1|> t) = p_n(|x_1|^2>t^2) = \mu( y_1^2+y_2^2 > t^2 \|y\|_2^2 ) = \mu (y_1^2 + y_2^2 > \frac{t^2}{1-t^2} \sum_{j>1} y_j^2)\; 
% $$
% 
% We write $\sqrt{\frac{t^2}{1-t^2}} = \cot (\theta_0)$, then $p_n(|x_1|> t)$ is given by the integral
% $$
% p_n(|x_1|> t) = \int_{\lbrace y_1^2 + y_2^2 > \frac{t^2}{1-t^2} \sum_{j>1} y_j^2\rbrace}\prod_k e^{-y_k^2/2}\frac{dy_k}{\sqrt{2\pi}}\; .
% $$
% 
% We do the change of variable
% $$
% y_1 = \cot \theta \cos u \sqrt{\sum_{j>2}y_j^2} \; ,\; y_2 = \cot \theta \sin u \sqrt{\sum_{j>1}y_j^2}\; ,
% $$
% whose Jacobian is given by
% $$
% dy_1dy_2 = \left( \sum_{j> 1}y_j^2\right) \frac{\cos \theta}{\sin^3 \theta} d\theta du\; .
% $$
% We get
% $$
% p_n(|x_1|> t) = \int_{\R^{2N_n-2}\times [0,\theta_0]}\prod_{k>1} e^{-y_k^2/2(1+\cot^2 \theta)}\frac{dy_k}{\sqrt{2\pi}}\left( \sum_{j> 1}y_j^2\right)\frac{\cos \theta}{\sin^3\theta }d\theta\; .
% $$
% 
% We then do the change of variable $z_j = \sqrt (1+\cot^2 \theta) y_j = \frac{y_j}{\sin \theta }$,
% \begin{eqnarray*}
% p_n(|x_1|> t) &= \int_{0}^{\theta_0}\sin^{2N_n-3}\cos \theta d\theta \int_{\R^{2N_n-2}} \left( \sum_{j>1} z_j^2 \right) \prod_k e^{-z_k^2/2}\frac{dz_k}{\sqrt{2\pi}}\\
%  & = \sin\theta_0^{2N_n-2} = (1-t^2)^{N_n-1} \; .
% \end{eqnarray*}
% \end{proof}

\begin{proposition} There exists $C$ such that for all $n,q$,
$$
M_{n,q} \leq C \sqrt q \; .
$$
\end{proposition}

\begin{proof}
Let us compute the mean value with respect to $p_n$ of $\|.\|_{L^q}^q$. Let
$$
A_{n,q}^q = E(\|.\|_{L^q}^q) = \int_{S_n} \Big(\int_{S^3} |u(x)|^qdx\Big)dp_n(u)\; .
$$
We can reverse the order of the integrals :
$$
A_{n,q}^q = \int_{S^3}\Big( \int_{S_n} |u(x)|^pdp_n(u) \Big) dx = \int_{S^3} \int_{R_+} q\lambda^{q-1}p_n(|u(x)|> \lambda) d\lambda dx\; .
$$
With our particular basis $(f_{n,k})_k$, we get that $u$ is written
$$
u = \sum_k a_k f_{n,k}
$$
hence with $K_n(x) = \sqrt{\sum_k |f_{n,k}(x)|^2}$ and $\epsilon (x)$ the unit vector : 
$$
\epsilon_k(x) = \frac{f_{n,k}(x)}{K_n(x)}\; ,
$$
we get
$$
p_n(|u(x)|> \lambda) = p_n \Big(|\langle a, \epsilon(x)\rangle | > \frac{\lambda}{K_n(x)}\Big)
$$
and since $p_n$ is invariant by the action of $O(N_n)$, $\epsilon(x)$ can be replaced by $(1,0,\hdots, 0)$ :
$$
p_n(|u(x)|> \lambda) = p_n(|a_1| > \frac{\lambda}{K_n(x)}) \leq 2e^{-(N_n-1) \frac{\lambda^2}{2K_n(x)^2}} \; .
$$
We already proved that $K_n(x) = C n \leq C \sqrt{N_n-1}$, hence 
$$
A_{n,q}^q \leq \int_{S^3}\left( \int_{\R_+}q\lambda^{q-1} 2e^{-\frac{\lambda^2}{2}}d\lambda\right) dx
$$
As we have by induction on $q$
$$
\int_{\R^+} q \lambda^{q-1} e^{-\lambda^2/ 2} \leq C q^{q/2}\; ,
$$
we get that
$$
A_{n,q} \leq C \sqrt q
$$
with $C$ independent from $n$ and $q$.

To bound $M_{n,q}$, we use the definition of the median : 
$$
\frac{1}{2} \leq p_n (\|u\|_{L^q} \geq M_{n,q})  = p_n (\|u\|_{L^q}^q \geq M_{n,q}^q) 
$$
and then using Markov's inequality,
$$
\frac{1}{2} \leq M_{n,q}^{-q}A_{n,q}^q\; .
$$
We deduce from that that
$$
M_{n,q} \leq 2^{1/q}A_{n,q} \leq C \sqrt q\; ,
$$
which concludes the proof.
\end{proof}

We know prove the existence of a sequence $p_m$ that goes to $\infty$ such that there exists an orthonormal basis 
$
(e_{n,k})_{n,k}
$
such that $e_{n,k}$ belongs to $E_n$ and 
$$
\|e_{n,k}\|_{L^{p_m}} \leq C \sqrt{p_m}
$$
where $C$ is independent from $n,k$ and $m$.

We have that for some constant $C$ independent from $n$ and $p$, the set
$$
B_{n,p} = \lbrace (e_{n,k})_{1\leq k\leq (n+1)^2} \in U_n \; | \; \forall k \|e_{n,k}\|_{L^p} \leq C \sqrt p \rbrace
$$
satisfies that
$$
\nu_n (B_{n,p}^c) \leq c_0 n^2 e^{-c'_1 C^2 n^{4/p} p}
$$
where $B_{n,p}^c$ is the complementary set of $B_{n,p}$ by taking $\Lambda = C \sqrt p$ in the Proposition \ref{prop-lip}. By taking the product measure $\nu$ of the $\nu_n$ and with
$$
B_p = \prod_n B_{n,p}
$$
we get
$$
\nu(B_p^c ) \leq c_0 \sum_{n\geq 1} n^2e^{-c'_1C^2 n^{4/p}p} \leq c_0 e^{-c_1'C^2 p} \sum_{n\geq 1}n^{2-4c_1'C^2}\; .
$$
Hence, for $C$ large enough, we have for all $p$:
$$
\nu(B_p^c) \leq \frac{1}{2}\; .
$$
By Fatou lemma, 
$$
\nu (\limsup_{p\rightarrow \infty } B_{p} ) \geq \frac{1}{2}
$$
Therefore, the set $\limsup_{p\rightarrow \infty } B_{p}$ is not empty. This is equivalent to the existence of a sequence $p_m \rightarrow \infty$ and a basis $e_{n,k}$ of spherical harmonics such that
\begin{equation}
\|e_{n,k}\|_{L^{p_m}} \leq C\sqrt{p_m} \; .
\end{equation}

\bibliographystyle{amsplain}
\bibliography{NPSbib} 
\nocite{*}

\end{document}